[1994/12/01]% LaTeX date must December 1994 or later
\documentclass[reqno,13pt]{amsart}
\usepackage{bbm}
\def\1{\mathbbm{1}}

\usepackage{amsmath,amsthm}
\usepackage[mathscr]{euscript}
\usepackage{graphicx}
\usepackage{amsmath,amssymb}
\usepackage{circuitikz}

\usepackage{amsmath}

\usepackage{amsfonts}

\usepackage{hyperref}

\hypersetup{
	colorlinks=true,                         
	linkcolor=red, % Couleur des liens internes
	citecolor=blue, % Couleur des num�ros de la biblio dans le corps
	urlcolor=red  } % Couleur des url

\usepackage{float}

\usepackage[percent]{overpic}

\usepackage{multicol}

\usepackage{subfloat}

\usepackage{graphics}

\usepackage{subcaption}

\usepackage{dblfloatfix}

\usepackage{fixltx2e}

\usepackage{latexsym}

\usepackage{fancybox}

\usepackage[utf8]{inputenc}

\usepackage{amssymb}

\usepackage{fancyhdr}

\usepackage{amsthm}

\usepackage{cases}

\usepackage{tikz}

\newtheorem{theorem}{Theorem}[section]

\newtheorem{lemma}{Lemma}[section]

\newtheorem{remark}{Remark}[section]

\newtheorem{definition}{Definition}[section]

\newtheorem{proposition}{Proposition}[section]

\newtheorem{mainassumptions}{ Assumptions}[section]

\numberwithin{equation}{section}

\title{Null Controllability for a Degenerate Structured Population Model}

\author{Yacouba Simpore,  Yassine El gantouh and  Umberto Biccari}

\thanks{The authors wish to thank Prof. Enrique Zuazua for his comments and suggestions.}

\address{Yacouba Simpore, Laboratoire LAMI, Universit\'e  Joseph Ki-Zerbo,01 BP 7021 Ouaga 01, Burkina Faso; simplesaint@gmail.com}
\address{Yassine El gantouh, Departamento de Matem\'{a}ticas, Universidad Aut\'{o}noma de Madrid, 28049 Madrid, Spain, elgantouhyassine@gmail.com}

\address{Umberto Biccari, Chair of Computational Mathematics, Fundación Deusto, Universityof Deusto, 48007 Bilbao, Basque Country, Spain}

\thanks{This work has been supported by the European Research Council (ERC) under the European Union’s Horizon 2020 research and innovation programme (grant agreement No 694126-DYCON). The
	work of U.B. is supported by the Grant PID2020-112617GB-C22 KILEARN of MINECO (Spain).}





\begin{document}

	\maketitle
	
	\maketitle

	\renewcommand{\sectionmark}[1]{}
	
	\begin{abstract}
		In this paper, we consider the infinite dimensional linear control system describing population models structured by age, size, and spatial position. The diffusion coefficient is degenerate at a point of the domain or both extreme points. Moreover, the control is localized in the space variable as well as with respect to the age and size. For each control support, we give an estimate of the time needed to control the system to zero. We establish the null controllability of the model by using a technique that avoids the explicit use of parabolic Carleman estimates. Indeed, our argument relies on a method that combines final-state observability estimates with the use of the characteristic method.
	\end{abstract}

	\section{Well-posedness and mains results}
	It has been recognized that age structure alone is not adequate to explain the population dynamics of some species (see e.g. \cite{b9,o,kk,z}). The size of individuals could also be used to distinguish cohorts. In principle there are many ways to differentiate individuals addition to age, such as body size and dietary requirements or some other physiological variables and behavioral parameters. For the sake of simplicity and the reason of similarity of mathematical treatment we assume here that only one internal variable is involved. Meanwhile, we consider the velocity of internal variable to be constant. It should be noted that this assumption is not restrictive, as it has been pointed out that the problem where the growth of an internal variable does not increase at the same rate as age can be transformed into the constant rate case.

	In this article, we are interested in the controllability properties of one-dimensional, degenerate, size- and age-structured population models described as (see e.g. \cite{b9})
	\begin{equation}
		\left\lbrace\begin{array}{ll}
			\dfrac{\partial y}{\partial t}+\dfrac{\partial y}{\partial a}+\dfrac{\partial y }{\partial s}-k(x)\dfrac{\partial^2 y}{\partial x^2}-b(x)\dfrac{\partial y}{\partial x}+\mu(a,s)y=mu, &\hbox{ in }Q ,\\ 
			y(0,a,s,t)=y(1,a,s,t)=0, &\hbox{ on }\Sigma,\\ 
			y\left( x,0,s,t\right) =\displaystyle\int\limits_{0}^{A}\beta(a)y(x,a,s,t)da, &\hbox{ in } Q_{S,T} \\
			y\left(x,a,s,0\right)=y_{0}\left(x,a,s\right),&
			\hbox{ in }Q_{A,S};\\
			y(x,a,0,t)=0,& \hbox{ in } Q_{A,T},
		\end{array}\right.
		\label{2}
	\end{equation}
	where $Q_{A,S}:=(0,1)\times (0,A)\times (0,S)$, $Q:=(0,1)\times (0,A)\times (0,S)\times(0,T)$, $\Sigma:=(0,A)\times (0,S)\times(0,T)$, $Q_{S,T}:=(0,1)\times (0,S)\times (0,T)$,  and $Q_{A,T}:=(0,1)\times (0,A)\times (0,T)$. Here $y(x,a,s,t)$ represents the population density of certain species of age $a\in (0,A)$ and size $s\in (0,S)$ at time $t\geq 0$ and location $x\in (0,1)$, where $A>0$ and $S>0$ are the maximal age of life and the maximal size of individuals, respectively. The age-dependent function $\beta$ denote the natural fertility and hence the formula $\int_{0}^{A}\beta(a)yda $ determine the density of newborn individuals of size $s$ at time $t$ and location the point $x$. The age- and size-dependent function $\mu$ denotes the death rate and we assume that it satisfies $\mu(a,s)=\mu_1(a)+\mu_2(s)$. The space-dependent diffusion coefficient $k$ can degenerate at one point of the domain or at both extreme points. For $0\leq a_1<a_2\leq A$ and $\omega=(l_1,l_2)\subset (0,1)$, we denotes by $m$ the characteristic function of $Q_{1}:=\omega\times (a_1,a_2)\times (s_1,s_2)$, which is the region where the control $u$ acts. The initial distribution of the population is $y_{0}\left(x,a,s\right)$. For more details about the modeling of such system we refer to G. Webb \cite{b9}.

	All along this paper, we assume that the fertility rate $\beta$ and the mortality rate satisfies the demographic properties:
	
	\begin{align*}
		(	{\bf H1}):&
		\left\lbrace
		\begin{array}{l}
			\mu_1(a)\geq 0 \text{ for every } a\in (0,A)\\
			\mu_1\in L^{1}\left([0,a^*]\right)\hbox{ for every }\; a^*\in [0,A) \\
			\displaystyle\int\limits_{0}^{A}\mu_1(a)da=+\infty
		\end{array} \right.
		,\quad
		(	{\bf H2}):
		\left\lbrace\begin{array}{l}
			\mu_2(s)\geq 0 \hbox{ for every } s\in (0,S)\\
			\mu_2\in L^{1}\left([0,s^*]\right)\hbox{ for every }\hbox{ } s^*\in [0,A) \\
			\displaystyle\int\limits_{0}^{S}\mu_2(s)ds=+\infty
		\end{array} \right.\cr
		(	{\bf H3}):&
		\left\lbrace\begin{array}{l}
			\beta(\cdot)\in C\left([0,A]\right)\cr 
			\beta(s) \ge 0 \quad {\rm  for \; a.e.} \quad a\in (0,A),\\
		\end{array}
		\right. \qquad
		(	{\bf H4}):
		\begin{array}{c}
			\beta(a)=0 \quad \forall a\in (0,\hat{a}) \text{ for some } \hat{a}\in (0,A). 
		\end{array}
	\end{align*}
	
	For more details about the modeling of such system and the biological significance of the hypotheses, we refer to G. Webb \cite{b9}.

	\textbf{NB :} All the figures on this paper represented in dimension two are in reality in dimension three; They are cubic in shape seen from above. The height represents the time $t$. The characteristics evolve from top to bottom but in the positive direction with respect to the age variable $a$ and the size variable $s.$
	
	\begin{figure}[H]
		
		\begin{center}
			
			\begin{overpic}[scale=0.25]{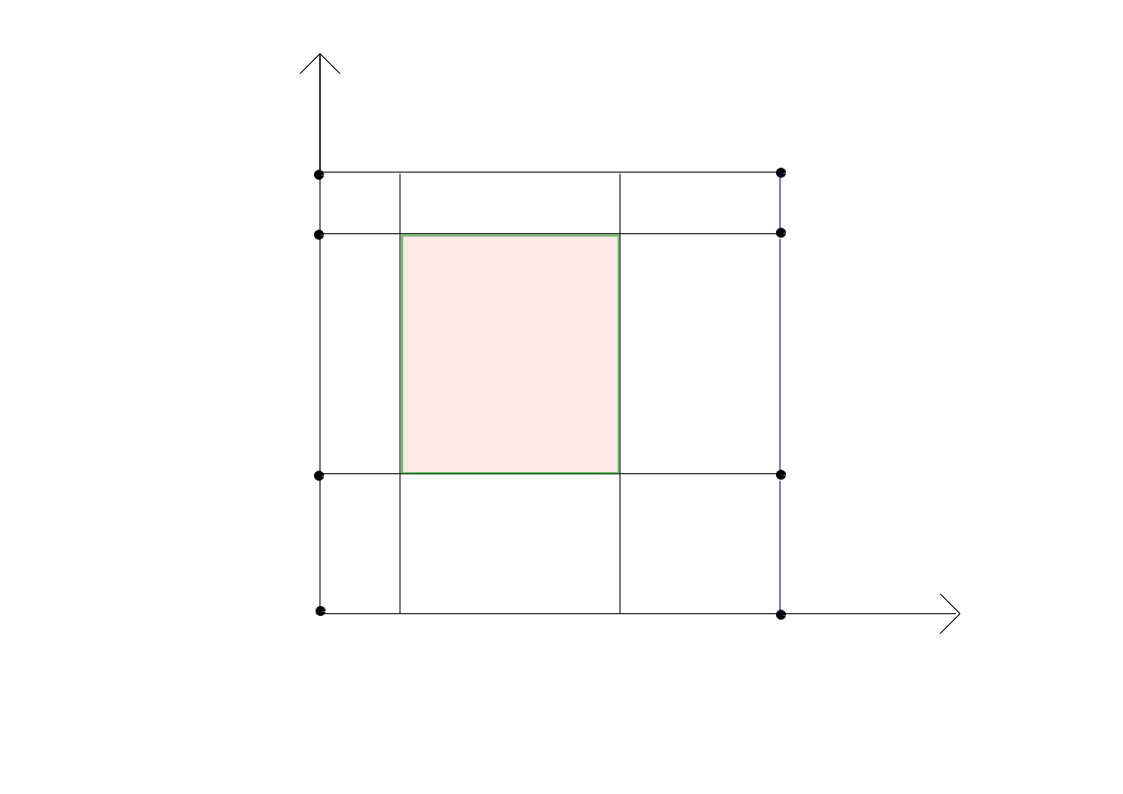}
				
				\put (21,53) {$S$}
				
				\put (22,49.5) {$s_2$}
				
				\put (40,34.5) {$Q_1$}	
				
				\put (22,27.5) {$s_1$}
				
				\put (35,12) {$a_1$}
				
				\put (52,12) {$a_2$}
				
				\put (67,10.5) {$A$}
				
			\end{overpic}
			
			\caption{Here is the support $m$ without the space variable see from above}.
			
		\end{center}
		
	\end{figure}
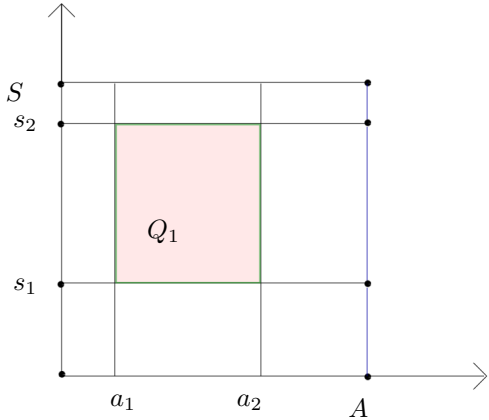
	
	Let $Q'$ denotes the domain $Q$ without the space variable, i.e., $Q'=Q\setminus(0,1)$. We split the domain $Q'$ as follow:
	
	\begin{align}\label{domain}
		A_1&:=\{(a,s,t)\in Q' \text{ such that } 0< t< A-a \text{ and } 0< t< S-s\}, \cr A'_1&:=\{(a,s,t)\in Q' \text{ such that } S-s>t>A-a>0\text{ or } t>S-s>A-a>0\},\cr
		A'_2&:=\{(a,s,t)\in Q' \text{ such that } A-a>t>S-s>0\text{ or } t>A-a>S-s>0\}.
	\end{align}
	
	To simplify notation we denote by $A_2=A'_1\cup A'_2$ and hence $Q'=A_1\cup A_2$. The figures below illustrate the domain $Q'$:
	
	\begin{figure}[H]
		
		\begin{subfigure}{0.4\textwidth}
			
			\begin{overpic}[scale=0.25]{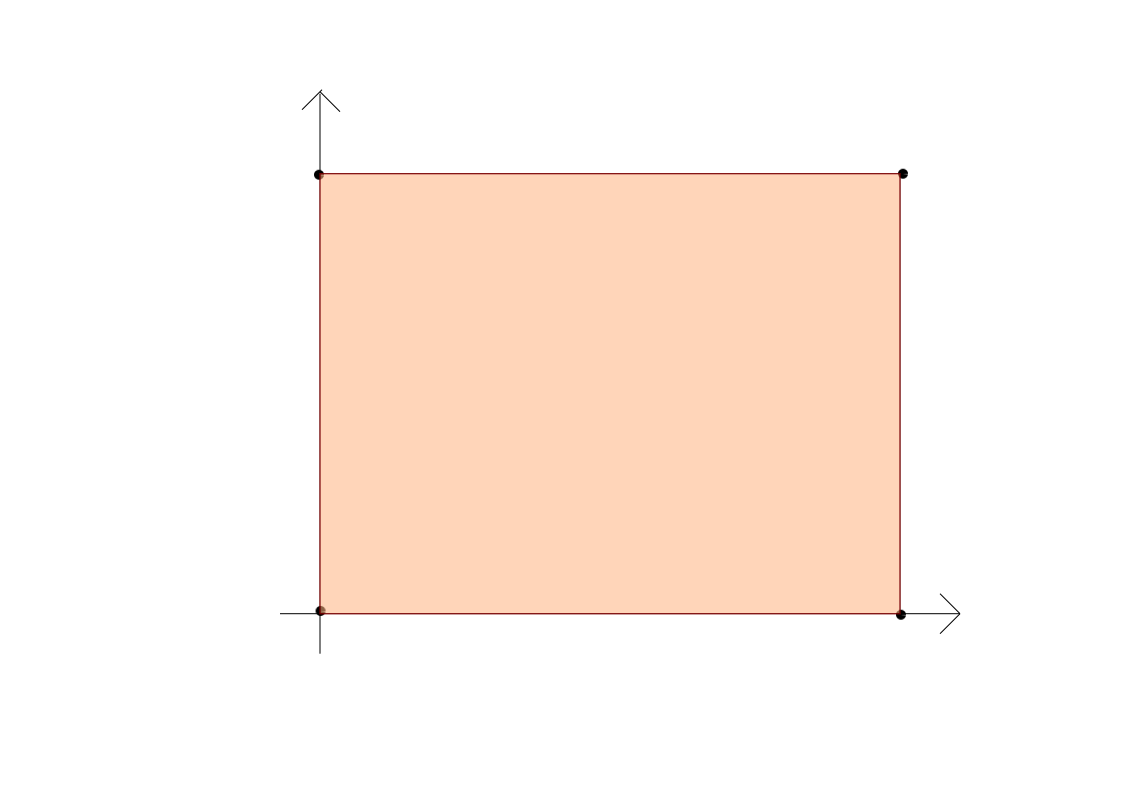}
				
				\put (21,53) {$S$}	
				
				\put (50,35) {$A_1$}
				
				\put (78.5,10.5) {$A$}
				
			\end{overpic}
			
			\subcaption{Here is the section $(t=0)$ of $Q'$}
			
		\end{subfigure}
		
		\quad
		
		\begin{subfigure}{0.5\textwidth}
			
			\begin{overpic}[scale=0.25]{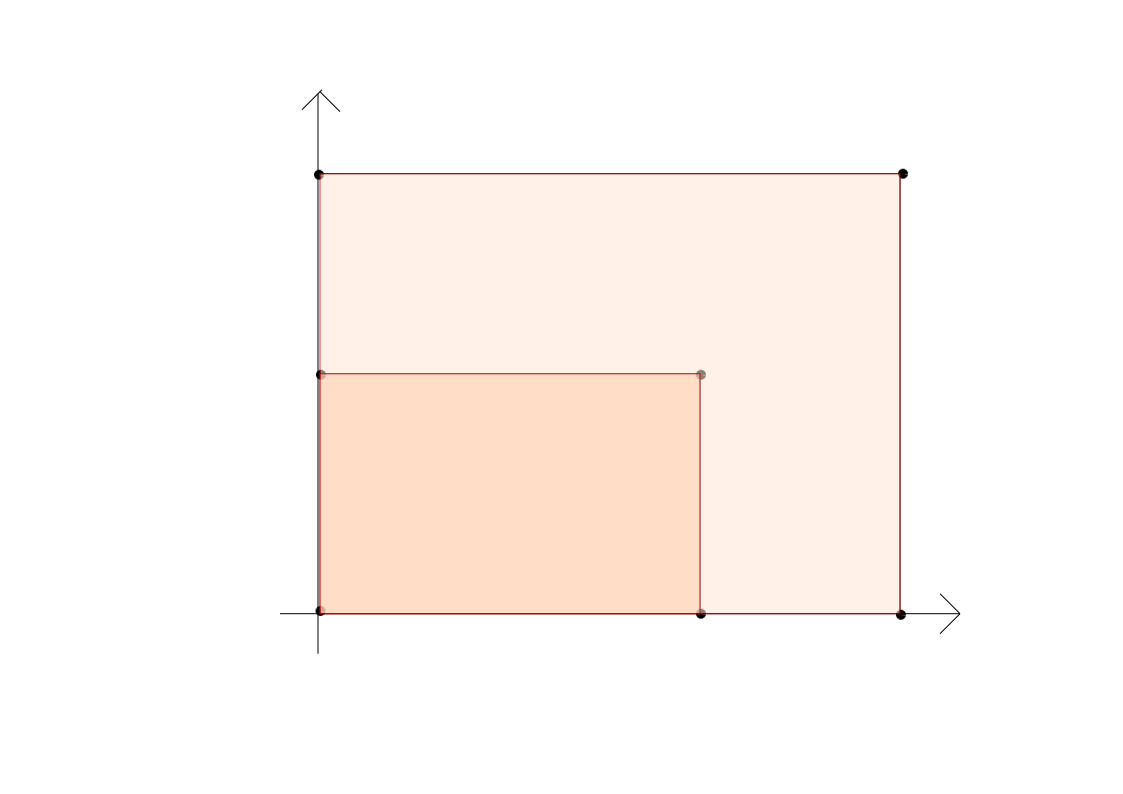}
				
				\put (22,53) {$S$}
				
				\put (42,25.5) {$A_1$}
				
				\put (52,42.5) {$A_2$}				
				
				\put (79,10.5) {$A$}
				
				\put (13,36.5) {$S-\alpha$}
				
				\put (55,10.5) {$A-\alpha$}	
				
			\end{overpic}
			
			\subcaption{Here is the section $(t=\alpha)$ of $Q'$ where $\alpha\in (0,\min\{A,S\}).$}
			
		\end{subfigure}
		
	\end{figure}
	
	\begin{center}
		
		\begin{figure}[H]
			
			\begin{overpic}[scale=0.25]{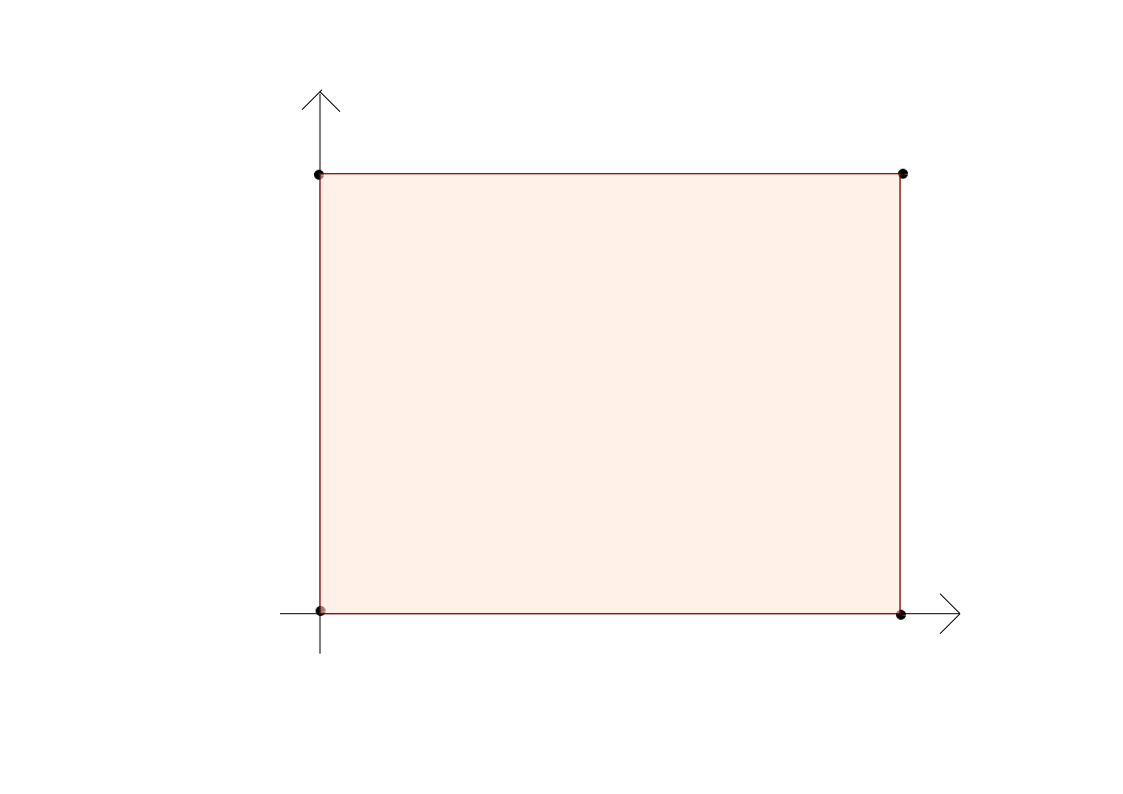}
				
				\put (20,53) {$S$}	
				
				\put (50,35) {$A_2$}
				
				\put (79.5,10.5) {$A$}	
				
			\end{overpic}
			
			\subcaption{Here is the section $(t=\alpha)$ where $\alpha\geq\min\{A,S\}$ of $Q'.$}
			
		\end{figure}
		
	\end{center}

	Before going further and stating the main result of this paper, we need to introduce some notations and assumptions. We start with the following assumptions on the diffusion coefficient $k$.
	
	\begin{mainassumptions}\label{Main1}
		
		To study well-posedness of \eqref{2} we assume that $k\in C([0,1])$ and there exist $M_1\hbox{, }M_2\in (0,2)$ such that \[k\in C^1(0,1)\hbox{, }k>0\hbox{ in }(0,1)\hbox{, }k(0)=k(1)=0,\]
		
		\[xk'(x)\leq M_1k(x)\hbox{ and }(x-1)k'(x)\leq M_2k(x)\]
		
		for all $x\in[0,1],$ or \[k\in C^1((0,1])\hbox{, }k>0\hbox{ in }(0,1]\hbox{, }k(0)=0\hbox{ and }xk'(x)\leq M_1k(x)\]
		
		for all $x\in[0,1],$ or
		
		\[k\in C^1((0,1])\hbox{, }k>0\hbox{ in }[0,1)\hbox{, }k(1)=0\hbox{ and }(x-1)k'(x)\leq M_2k(x)\]
		
		for all $x\in[0,1]$. On the other hand, \[b\in C([0,1])\cap C^2(0,1)\hbox{ and }\dfrac{b}{k}\in L^1(0,1).\] 
		
	\end{mainassumptions}
	
	Next, let us consider the weight function $\gamma$ defined by (cf. \cite{12})
	
	$$ \gamma(x):=\exp\left(\int\limits_{x}^{\frac{3}{4}}\dfrac{b(y)}{k(y)}dy\right)\qquad x\in[0,1],$$
	
	and define 
	
	$$\sigma(x):=k(x)\dfrac{1}{\gamma(x)}, \qquad x\in[0,1].$$
	
	We then select the following weighted Lebesgue and Hilbert spaces
	
	\begin{align*}
		L^{2}_{\frac{1}{\sigma}}(0,1)&:=\{u\in L^2(0,1)|\;\|u\|_{\frac{1}{\sigma}}<+\infty\},\qquad\|u\|^{2}_{\frac{1}{\sigma}}=\int\limits_{0}^{1}u^2\frac{1}{\sigma}dx,\cr
		H^{1}_{\frac{1}{\sigma}}(0,1)&:=L^{2}_{\frac{1}{\sigma}}(0,1)\cap H^{1}_{0}(0,1) \qquad\|u\|^{2}_{1,\frac{1}{\sigma}}=\|u\|^{2}_{\frac{1}{\sigma}}+\int\limits_{0}^{1}u^{2}_{x}dx,
	\end{align*}
	
	For $u$ sufficiently smooth, e.g. $u\in W^{2,1}_{loc}(0,1)$, we set
	
	$$Lu:=ku_{xx}-b(x)u_{x},$$ 
	
	for almost every $x\in (0,1)$. Moreover for all $(u,v)\in H^{2}_{\frac{1}{\sigma}}(0,1)\times H^{1}_{\frac{1}{\sigma}}(0,1),$ 
	
	$$(Lu,v)_{\frac{1}{\sigma}}=-\int\limits_{0}^{1}\gamma u_x v_x dx,$$
	
	where
	
	\[H^{2}_{\frac{1}{\sigma}}(0,1):=\{u\in H^{1}_{\frac{1}{\sigma}}(0,1)|L u\in L^{2}_{\frac{1}{\sigma}}(0,1)  \} \quad \|u\|^{2}_{2,\frac{1}{\sigma}}=\|u\|^{2}_{1,\frac{1}{\sigma}}+\|L u\|^{2}_{\frac{1}{\sigma}}.\]

	On the other hand, to establish the null controllability result of \eqref{2}, we additionally assume that $k$ satisfies the following assumptions (see, e.g. \cite{canard1})
	
	\begin{mainassumptions}\label{Main2}
		The function $k\in C^0([0,1])\cap C^3(0,1)$ is such that satisfies $k(0)=k(1)=0,$ and $k>0\hbox{ in }(0,1).$\\ Moreover,
		\begin{enumerate}
			
			\item there exists $\epsilon\in (0,1)\hbox{, }N_1\in (0,2)\hbox{ and }C_1>0$ such that the function \[\dfrac{x(b-k_x)}{k}\in L^{\infty}(0,\epsilon)\hbox{, }x(k)_x\leq N_1k(x)\hbox{ for all }x\in (0,\epsilon),\]
			
			there exists a function \[C_1=C_1(\epsilon)>0\hbox{ defined in }(0,\epsilon)\hbox{ such that }C_1(\epsilon')\longrightarrow 0\hbox{ as }\epsilon'\longrightarrow 0^+\]
			
			and
			
			\[\left|\left(\dfrac{x(b-k_x)}{k}\right)_{xx}-\dfrac{b}{a}\left(\dfrac{x(b-a_x)}{a}\right)\right|\leq \dfrac{C_1(\epsilon')}{x^2}\hbox{ for all }x\in(0,\epsilon').\]
			
			\item
			
			\[\dfrac{(x-1)(b-k)_x}{k}\in L^{\infty}(1-\epsilon,1)\hbox{ and there exists } N_2\in(0,2) \hbox{ and }C_2>0\hbox{ such that },\]

			\[\left(\dfrac{(x-1)(k)_x}{k}\right)_{xx}<\dfrac{C_2}{k}\hbox{ for all }x\in (1-\epsilon,1),\]
			
			and there exists a function \[C_2=C_2(\epsilon)>0\hbox{ defined in }(1-\epsilon,1)\hbox{ such that }C_1(\epsilon')\longrightarrow 0\hbox{ as }\epsilon'\longrightarrow 0^+\]
			
			\[\left|\left(\dfrac{(x-1)(b-k_x)}{k}\right)_{xx}-\dfrac{b}{a}\left(\dfrac{(x-1)(b-a_x)}{a}\right)\right|\leq \dfrac{C_1(\epsilon')}{(x-1)^2}\hbox{ for all }x\in(1-\epsilon',1).\]
			
		\end{enumerate}
		
	\end{mainassumptions}

	After such a long but necessary preparation we can now clearly state the main results of this paper. The well-posedness of \eqref{2} follows from the following result.
	
	\begin{theorem}\label{MainT1}
		
		Let the conditions {\bf (H1)}-{\bf (H3)} be satisfied. Furthermore, let assume that Assumption \ref{Main1} holds. Let $f\in L^{2}_{\frac{1}{\sigma}}(Q)$ and $u_0\in L^{2}_{\frac{1}{\sigma}}(Q_{A,S})$, then $(\ref{2})$ has a unique solution \[u\in \mathcal{U}:=C\left([0,T],L^{2}_{\frac{1}{\sigma}}(Q_{A,S})\right) \cap L^2\left(0,T;H^{1}(0,A)\times H^1(0,S) \times H^{1}_{\frac{1}{\sigma}}(0,1)\right).\]
		If, in addition, $f=0$, then 
		\begin{align*}
			u\in C^1\left([0,T];L^{2}_{\frac{1}{\sigma}}(Q_{A,S})\right).
		\end{align*}
		
	\end{theorem}

	We select the following definition.
	
	\begin{definition}
		
		The real $\hat{a}$ is the minimal age from which individuals become fertile. We will call it the \emph{ minimal age of fertility}.
		
	\end{definition}
	
	For $0 \leq s_1 \leq s_2 \leq S $, let denote 
	
	$$T_0=\max\{s_1, S-s_2\} \hbox{ and } T_1=\max\{s_1,a_1+S-s_2\}.$$

	Thus, our main result is the following.
	
	\begin{theorem}\label{MainT2}
		
		Let the assumptions {\bf (H1)}-{\bf (H4)} be satisfied and assume that the Assumptions \ref{Main1}-\ref{Main2} hold. Furthermore, we assume that $a_1< \hat{a}$, $T_0<\min\{a_2-a_1,\hat{a}-a_1\}$. Then for every $T>A-a_2+a_1+S-s_2+s_1$ and for every $y_0\in L^{2}_{\frac{1}{\sigma}}(Q_{A,S}) $
		
		there exists a control $u_1\in L^{2}_{\frac{1}{\sigma}}(Q_1)$ such that the solution $y$ of \eqref{2} satisfies 
		
		\[ y(x,a,s,T)=0\hbox{ a.e }x \in (0,1)\hbox{ } a \in (0,A)\hbox{ and } s \in (0,S).\]	
		
	\end{theorem}

	We now depict some related works in the literature. Many versions of age-size structured models, both linear and nonlinear, have been investigated, and seminal treatments of such models are given by Metz and Diekmann \cite{j} and Tucker and Zimmerman \cite{z}. Spatial structure in linear and nonlinear age or size structured models has attracted many interest in these last years. In \cite{j,b9,z}, the authors have studied, using semigroup theory, the existence and uniqueness of solutions. The null controllability of the Lotka-McKendrick system with or without spatial diffusion have been addressed by several researchers (see, e.g. \cite{yac3,dy,b1,B1,abst,b6,B} and references therein). The first null controllability result was obtained by Ainseba and Anita \cite{b1}. They showed that the so-called Lotka-Mckendrick model with spatial diffusion can be driven to a steady state in any time $T>0$, preserving the positivity of the trajectory, provided that the initial data are close to the steady state and the control acts in a spatial subdomain $\omega\subset\Omega$, but for all ages. Recently, Maity \cite{b6} proved that null controllability can be achieved by controls supported in any subinterval of $(0,A)$, provided we control before individuals start to reproduce. In \cite{b3}, Hegoburu and Tucsnak proved that the same system is null controllable for all ages and in any time by controls localized with respect to the spatial variable but active for all ages. In \cite{Genni} the authors studied the null controllability property for a single population model in which the population $y$ depends on time $t$ space $x,$ age $a$ and size $\tau$ by using Carleman estimates. It should be noted that in their work the growth function $g$ is variable and depends only on the size $s$. More recently, D. Maity, M. Tucsnak and E. Zuazua in \cite{dy} solved the problem of null controllability of a Lotka-Mckendrick model with spatial diffusion, where the control is localized in the space variable as well as with respect to the age. The method combines final-state observability estimates with the use of characteristics and $L^{\infty}$-estimates of the associated semigroup. The same method is used in \cite{yac2}, to establish the null controllability of a nonlinear Lotka-mckendrick system with respect to a control localized in the space variable as well as on the age. The authors Y.Simpor\'e and O.Traor\'e have recently in \cite{yac1} solved a problem of null controllability of a nonlinear age, space and two-sex structured population dynamics model. Since the system is a cascade system and the coupling is at the renewal term level, the approach developed in \cite{dy} was the most suitable to achieve the desired control objective.
	
	This paper addressee the null controllability of a model of population dynamics structured with age, size, and spatial diffusion. Indeed, it has been recognized that age structure alone is not sufficient to explain the population dynamics of some species. The size of individuals could also be used to distinguish cohorts.  Hence the interest in this work.
	
	The main novelties brought in by our paper are:
	
	$\bullet$ In this paper, we revisit the work of \cite{canard} with the growth function $g(s)=1$, but with a control localized in space, age and size. For each control support $m_i,$ we establish an estimate of the time $T$ depending on $m_i$, necessary to control the system at zero using some assumptions on the kernel $\beta.$
	
	$\bullet$ In this paper, we use the approach developed in \cite{dy} which combine final-state observability estimates with the use of the characteristic method. This technique is interesting because it avoids the establishment of Carleman inequality, which seems to be very expensive. In fact, with this technique we obtain a global controllability result that applies to individuals of all ages and sizes without having to exclude ages and sizes in a neighborhood of zero, which is not the case with Carleman's estimates.
	
	The remaining part of this work is organized as follows:
	In Section \ref{Sec2}, we establishes the final state observability on space, age and size for the adjoint system and, as a consequence, we obtain the proof of the main result (Theorem \ref{MainT2})

	\section{An Observability Inequality}\label{Sec2}
	
	It is well-known that the system \eqref{2} can equivalently be rewritten as an abstract evolution system:
	
	\begin{align}\label{Abstract}
		\begin{cases}
			\frac{\partial}{\partial t} y(t)=\mathcal{A}y(t)+\mathcal{B}u(t),& t\geq 0,\cr
			y(0)=y_0,
		\end{cases}
	\end{align}
	
	also to be referred to as $(\mathcal{A},\mathcal{B})$ in the sequel, where we can identify the operators $\mathcal{A}$ and $\mathcal{B}$ through their adjoints by formally taking the inner product of \eqref{Abstract} with a smooth function $\varphi$ see, for instance, \cite[Section 11.2]{b8}. The state and control spaces are
	
	\begin{align*}
		K:=L^{2}_{\frac{1}{\sigma}}(Q_{A,S}),\qquad U:= L^{2}_{\frac{1}{\sigma}}(Q_1).
	\end{align*}
	
	The unbounded linear operator $\mathcal{A}:D(\mathcal{A})\subset K\longrightarrow K$, is defined, for every $y\in D(\mathcal{A})$ by
	
	$$\mathcal{A}y=L y-\dfrac{\partial y}{\partial a}-\dfrac{\partial y}{\partial s}-\left(\mu_1(a)+\mu_2(s)\right)y,$$
	
	with domain
	
	$$D(\mathcal{A}):= \begin{cases}
		y\in L^2((0,A)\times (0,S);H^{2}_{\frac{1}{\sigma}}(0,1)):\quad  \mathcal{A}y\in L^{2}_{\frac{1}{\sigma}}(Q_{A,S}),\cr
		\dfrac{\partial y}{\partial a},\dfrac{\partial y}{\partial s} \in L^2((0,A)\times (0,S);H^{1}_{\frac{1}{\sigma}}(0,1)),\cr
		y(1,\cdot,\cdot)=y(0,\cdot,\cdot)=0,\cr
		y(x,a,0)=0,\quad y(x,0,s)=\int_{0}^{A}\beta(a)y(x,a,s,t)da\quad {\rm for \; a.e}\quad x\in (0,1).
	\end{cases}$$

	A similar argument as in \cite{Genni} yields the following generation result.
	
	\begin{lemma}\label{generation-result}
		Let the conditions {\bf (H1)}-{\bf (H3)} be satisfied and let assume that Assumption \ref{Main1} holds. Then the operator $(\mathcal{A},D(\mathcal{A}))$ generates a C$_0$-semigroup $\mathfrak{T}:=(e^{t\mathcal{A}})_{t\geq 0}$ on $K$.
	\end{lemma}
	
	As such, a similar computation as in \cite{dy} shows that
	
	\begin{align*}
		D(\mathcal{A}^*):= \begin{cases}
			\varphi\in K,\quad \varphi\in L^2((0,A)\times (0,S);H^{2}_{\frac{1}{\sigma}}(0,1)),\cr
			\dfrac{\partial y}{\partial a}+\dfrac{\partial y}{\partial s}+L\varphi-\left(\mu_1(a)+\mu_2(s)\right)\varphi \in K,\cr
			\varphi(1,\cdot,\cdot)=\varphi(0,\cdot,\cdot)=0,\cr
			\varphi(\cdot,A,\cdot)=0,\quad \varphi(\cdot,S,\cdot)=0,
		\end{cases}
	\end{align*}
	
	and we have, for every $\varphi\in D(\mathcal{A}^*)$,
	
	\begin{align*}
		\mathcal{A}^*\varphi=L \varphi+\dfrac{\partial \varphi}{\partial a}+\dfrac{\partial \varphi}{\partial s}-\left(\mu_1(a)+\mu_2(s)\right)\varphi+\beta(a)\varphi(x,0,s),
	\end{align*}
	
	On the other hand, the control operator $\mathcal{B}\in \mathcal{L}(U,K)$ is given for every $u\in U$
	
	$$\mathcal{B}u=mu.$$

	With the above notation, one can see that the adjoint problem of \eqref{2} is given by:
	
	\begin{equation}
		\left\lbrace\begin{array}{ll}
			\dfrac{\partial q}{\partial t}-\dfrac{\partial q}{\partial a}-\dfrac{\partial q}{\partial s}-L q+(\mu_1(a)+\mu_2(s))q=\beta(a)q(x,0,s,t)&\hbox{ in }Q ,\\ 
			q(0,a,s,t)=q(1,a,s,t)=0&\hbox{ on }\Sigma,\\ 
			q\left( x,A,s,t\right) =0&\hbox{ in } Q_{S,T} \\
			q(x,a,S,t)=0& \hbox{ in } Q_{S,T}\\
			q\left(x,a,s,0\right)=q_0(x,a,s)&
			\hbox{ in }Q_{A,S};
		\end{array}\right.
		\label{112}
	\end{equation}
	where we recall that the operator $(L,D(L))$ is defined by 
	\begin{align*}
		Ly:=k(x)\dfrac{\partial^2 y}{\partial x^2}+b(x)\dfrac{\partial y}{\partial x},\qquad y\in D(L):=H^{2}_{\frac{1}{\sigma}}(0,1)\cap H^{1}_{\frac{1}{\sigma}}(0,1).
	\end{align*}
	Next, let us consider the operator $P$ defined by
	\begin{equation}\label{33}
		P:=-\mu_1(a)-\mu_2(s)+L, \qquad {\rm on }\qquad D(\mathcal{A}^*),
	\end{equation}
	
	and denotes by $(e^{tP})_{t\geq 0}$ its associated C$_0$-semigroup. 
	
	%It is clear that the operator $P$ is the generator of a strongly continuous  on $K$.

	Thus, using the method of characteristics, we get the following result.
	
	\begin{proposition}
		Let the assumptions of Lemma \ref{generation-result} be satisfied. Then, for any initial data $q_0\in K$, the adjoint system \eqref{112} admits a unique solution $q$ given by
		\begin{align}\label{alp}
			q(t) =\begin{cases}
				q_0(.,a+t,s+t)e^{t P}+\displaystyle\int\limits_{0}^{t}e^{(t-r)P}\beta(a+t-r)q(x,0,s+t-r,r)dr,& {\rm in } \; A_1, \\
				\displaystyle\int_{\max\{t-S+s,t-A+a\}}^{t}e^{(t-r)P}\beta(a+t-r)q(x,0,s+t-r,r)dr, & {\rm in } \; A_2,
			\end{cases}
		\end{align}
		where we recall that $A_1,A_2$ are defined in \eqref{domain}, and we have, for every $t\geq 0 $,
		\[e^{tP}=\dfrac{\pi_1(a)}{\pi_1(a-t)}\dfrac{\pi_2(s)}{\pi_2(s-t)}e^{tL} \hbox{ with } \pi_1(a)=e^{-\int_{0}^{a}\mu_1(r)dr}\hbox{ and   }\pi_2(s)=e^{-\int_{0}^{s}\mu_2(r)dr}.\]
	\end{proposition}	
	
	\begin{proof}	
		The proof of the existence and uniqueness of the mild solution of \eqref{112} follows from Lemma \ref{generation-result}.
		Let us consider the following function 
		$$w(x,\lambda)=q(x,a+t-\lambda,s+t-\lambda,\lambda).$$ 
		Then, $w$ satisfies 
		\begin{align}
			\left\lbrace\begin{array}{l}
				w'(x,\lambda)=(-\mu_1(a+t-\lambda)-\mu_2(s+t-\lambda)+L) w(\lambda,x)+ f(x,\lambda),\\
				w(1,\lambda)=w(0,\lambda)=0,\\
				w(x,0)=q(x,a+t,s+t,0),
			\end{array}\right.
		\end{align}
		
		where $f(\lambda,x)=\beta(a+t-\lambda)q(x,0,s+t-\lambda,\lambda)$. The solution of the homogeneous equation is given by
		\[w_H(x,\lambda)=Ce^{\lambda P}.\]
		Notice that $$q(x,a,s,t)=w(x,t).$$
		
		\textbf{Taking into account the initial condition}
		
		Here we consider the domain $A_1$. The Duhamel formula yields that
		\[w(x,t)=e^{tP}w(x,0)+\int_{0}^{t}e^{(t-r)P}\beta(a+t-r)q(x,0,s+t-r,r)dr.\]
		Moreover, since
		\[w(x,0)=q(x,a+t,s+t,0)=q_0(x,a+t,s+t),\]
		then
		$$q(x,a,s,t)=q_{0}(x,a+t,s+t)e^{t P}+\int\limits_{0}^{t}e^{(t-r)P}\beta(a+t-r)q(x,0,s+t-r,r)dr, \quad {\rm in }\quad A_1.$$	
		
		\textbf{Taking into account the boundary condition in $a$}
		
		For the boundary condition in $\{a=A\}$ we use the set
		\[A'_1:=\{(a,s,t)\in Q' \text{ such that } S-s>t>A-a> 0\text{ or } t>S-s>A-a> 0\}.\]
		Then using the Duhalmel formula (boundary condition in age is $q(x,A,s,t)$) we obtain 
		\begin{equation}\label{souff}
			q(x,a,s,t)=e^{(A-a)P}w(x,t-(A-a))+ \int_{t-A+a}^{t}e^{(t-r)P}\beta(a+t-r)q(x,0,s+t-r,r)dr. 
		\end{equation}
		Since $$w(x,t-(A-a))=q(x,A,s-a+A,a+t-A)=0,$$ then 
		$$ q(x,a,s,t)=\int\limits_{t-A+a}^{t}e^{(t-r)P}\beta(a+t-r)q(x,0,s+t-r,r)dr .$$	
		
		\textbf{Taking into account of the boundary condition in $s.$}
		
		For the boundary condition in $\{s=S\}$ we use the set 
		\[A'_2:=\{(a,s,t)\in Q' \text{ such that } A-a>t>S-s> 0\text{ or } t>A-a>S-s> 0\}.\] Then
		using the Duhamel formula, we obtain:
		$$q(x,a,s,t)=e^{(S-s)r}w(t-(S-s),x)+\int\limits_{t+s-S}^{t}e^{(t-r)P}\beta(a+t-r)q(x,0,s+t-r,r)dr .$$
		It follows from $w(t+s-S,x)=q(x,a,S,t)=0$ that 
		\begin{equation}\label{souf}
			q(x,a,s,t)=\int_{t+s-S}^{t}e^{(t-r)P}\beta(a+t-r)q(x,0,s+t-r,r)dr. 
		\end{equation} 
	\end{proof}
	Since the operator $\mathcal{A}$ generates a C$_0$-semigroup on $K$ (see Lemma \ref{generation-result}) and that $\mathcal{B}$ is a bounded operator. It also follows that the abstract system \eqref{Abstract} is well-posed in the sense that: for every $y_0\in K$ and every $u\in L^2(0,+\infty;U)$, there exists a unique solution $y\in C(0,+\infty;K)$ to \eqref{Abstract} given by the Duhamel formula (see e.g. \cite[Proposition 4.2.5]{b8})
	\begin{align}
		y(T)=e^{T\mathcal{A}}y_0+\Phi_Tu,\qquad T\geq 0, 
	\end{align}
	where $\Phi_T$ is the so-called input map of $(\mathcal{A},\mathcal{B})$, that is the linear operator defined for every $u\in L^2(0,+\infty;U)$ by 
	\begin{align*}
		\Phi_Tu=\int_{0}^{T} e^{(T-s)\mathcal{A}}\mathcal{B}u(s)ds.
	\end{align*}
	We recall that $(\mathcal{A},\mathcal{B})$ (and hence \eqref{2}) is null controllable in time $T>0$ if ${\rm Ran}\, e^{T\mathcal{A}}\subset {\rm Ran}\,\Phi_T$ (see, e.g. \cite[Definition 11.1.1]{b8}). It is also well-known that the null controllability of the pair $(\mathcal{A},\mathcal{B})$ is equivalent to the final state observability of the pair $(\mathcal{A}^*,\mathcal{B}^*)$. More precisely (see e.g. \cite[Theorem 11.2.1]{b8}), $(\mathcal{A},\mathcal{B})$ is null controllable in time $T$ if, and only if, there exists $C_T>0$ such that
	\begin{equation}
		\displaystyle\int\limits_{0}^{T}\|\mathcal{B}^*e^{t\mathcal{A}^*}q_{0}\|^2\geq C_{T}^{2}\|e^{T\mathcal{A}^*}q_0\|^2, \qquad \forall\, q_{0}\in D(\mathcal{A}^*),
	\end{equation}
	where $\mathfrak{T}^{*}:=(e^{t\mathcal{A}^*})_{t\geq 0}$ is the adjoint semigroup of $\mathfrak{T}$ generated by $\mathcal{A}^{*}$. 
	\subsection{Proof of the main result}
	In this part we provide the proof of the main result of this paper, namely Theorem \ref{Main2}. Indeed, in view of \cite[Theorem 11.2.1]{b8} we will mainly perform computations on the adjoint system \eqref{112} in the sequel. Then, it is then more convenient to restate the result of Theorem \ref{Main2} as follows:
	
	\begin{theorem}\label{Main3}
		Let the assumptions of Theorem \ref{Main2} be satisfied. Then, for every $q_0\in D(\mathcal{A}^*)$, the pair $(\mathcal{A}^{*},\mathcal{B}^{*})$ is final-state observable in any time $T>A-a_2+S-s_2+a_1+s_1$. In particular, for every $T>A-a_2+S-s_2+a_1+s_1$ there exist $C_T>0$ such that the solution $q$ of \eqref{112} satisfies 
		\begin{equation}
			\int\limits_{0}^{S}\int\limits_{0}^{A}\int_{0}^1\dfrac{1}{\sigma(x)}q^2(x,a,s,T)dxdads\leq C_T\displaystyle\int\limits_{0}^{T}\int\limits_{s_1}^{s_2}\int\limits_{a_1}^{a_2}\int_{\omega}\dfrac{1}{\sigma(x)}q^2(x,a,s,t)dxdadsdt.
		\end{equation}
	\end{theorem}	
	
	For the proof of the above theorem we proceed as in \cite{abst}. The idea of the proof is based on the estimation of $\beta(a)q(x,0,s,t)$. Let us recall that $$T_0=\max\{s_1, S-s_2\} \hbox{ and } T_1=\max\{s_1,a_1+S-s_2\}.$$
	
	Hence, we have the following proposition. %Let $T_0<\min\{\hat{a}-a_1,a_2-a_1\}$, where we recall that $T_0=\max\{s_1, S-s_2\}$.
	
	\begin{proposition}\label{Prop1}
		Let the assumptions of Theorem \ref{Main2} be satisfied and let $a_1<\hat{a}$, $T_0<\min\{\hat{a}-a_1,a_2-a_1\}$ and $T_1<\eta<T$. Then there exists a constant $C>0$ such that for every $q_{0}\in K$, the solution $q$ of \eqref{112} satisfies the following inequality:
		\begin{equation}\label{Bq}
			\int\limits_{\eta}^{T}\int\limits_{0}^{S}\int_{0}^1\dfrac{1}{\sigma(x)}q^2(x,0,s,t)dxdsdt\leq C\displaystyle\int\limits_{0}^{T}\int\limits_{s_1}^{s_2} \int\limits_{a_1}^{a_2}\int\limits_{l_1}^{l_2}\dfrac{1}{\sigma(x)}q^2(x,a,s,t)dxdadsdt .
		\end{equation}
		More explicitly, 
		\begin{enumerate}
			\item If $s_1<a_1<s_2$, then for every $\eta_2$, with $a_1<\eta_1<T$ and for every $\delta>0$ such that $0<s_2-a_1-\delta$, there exists a constant $C_1>0$ such that for every $q_{0}\in K,$ the solution $q$ of \eqref{112} satisfies the following inequality:
			\begin{equation}\label{Bq12}
				\int\limits_{\eta_1}^{T}\int\limits_{0}^{s_2-a_1-\delta}\int_{0}^1\dfrac{1}{\sigma(x)}q^2(x,0,s,t)dxdsdt\leq C_3\displaystyle\int\limits_{0}^{T}\int\limits_{s_1}^{s_2} \int\limits_{a_1}^{a_2}\int\limits_{l_1}^{l_2}\dfrac{1}{\sigma(x)}q^2(x,a,s,t)dxdadsdt ;
			\end{equation} 
			and, $$q(x,0,s,t)=0\hbox{ a.e. in }(0,1)\times (s_2-a_1,S)\times (S-s_2+a_1,T);$$ 
			\item If $s_1>a_1$ and for every $\eta_3\hbox{ and }\eta_4$ such that $$a_1<\eta_3<T\hbox{, }s_1-a_0<\eta_4<T\hbox{ with }a_0\in [0,s_1-a_1)$$ and for every $\delta>0$ such that $0<s_2-a_1-\delta$ there exists $C_2>0\hbox{ and }C_3>0$ such that for every $q_{0}\in K,$ the solution $q$ of \eqref{112} verifies the following inequality:
			\begin{equation}
				\displaystyle\int\limits_{\eta_3}^{T}\int\limits_{s_1-a_1}^{s_2-a_1-\delta}\int_{0}^1\dfrac{1}{\sigma(x)}q^2(x,0,s,t)dxdsdt\leq C_2\displaystyle\int\limits_{0}^{T}\int\limits_{s_1}^{s_2} \int\limits_{a_1}^{a_2}\int\limits_{l_1}^{l_2}\dfrac{1}{\sigma(x)}q^2(x,a,s,t)dxdadsdt \label{Bq1};
			\end{equation} 
			and
			\begin{equation}
				\displaystyle\int\limits_{\eta_4}^{T}\int\limits_{a_0}^{s_1-a_1}\int_{0}^1\dfrac{1}{\sigma(x)}q^2(x,0,s,t)dxdsdt\leq C_3\displaystyle\int\limits_{0}^{T}\int\limits_{s_1}^{s_2} \int\limits_{a_1}^{a_2}\int\limits_{l_1}^{l_2}\dfrac{1}{\sigma(x)}q^2(x,a,s,t)dxdadsdt \label{Bq1};
			\end{equation} 
			\item else, if $s_2\leq a_1$, then $$q(x,0,s,t)=0\hbox{ a.e. in }(0,1)\times (0,S)\times (S-s_2+a_1,T);$$
		\end{enumerate}
	\end{proposition}
	
	To prove Proposition \ref{Prop1}, we recall the following observability inequality for parabolic equation (see, for instance, \cite{canard,canard1}):
	
	\begin{proposition}\label{observability-inequality}
		Let the Assumptions \ref{Main1}--\ref{Main2} be satisfied. Let $T>0$ and let $t_0,t_1>0$ such that $0<t_0<t_1<T$. Then, for every $w_0\in L^{2}_{\frac{1}{\sigma}}(0,1),$ the solution $w$ of the following initial boundary value problem 
		\begin{equation}
			\begin{cases}
				\dfrac{\partial w(x,\lambda)}{\partial \lambda}-L w(x,\lambda)=0, & \text{ in } (0,1)\times(t_0,T),\\
				w(0,\lambda)=w(1,\lambda)=0, & \text{ on } (t_0,T)	\\
				w(x,t_0)=w_0(x), &\text{ in }(0,1),
			\end{cases}
		\end{equation}
		satisfy 
		\[\displaystyle\int_{0}^1\dfrac{1}{\sigma(x)}w^2(T,x)dx\leq\int_{0}^1\dfrac{1}{\sigma(x)}w^2(x,t_1)dx\leq c_1e^{\dfrac{c_2}{t_1-t_0}}\int\limits_{t_0}^{t_1}\int_{\omega}\dfrac{1}{\sigma(x)}w^2(x,\lambda)dxd\lambda,\]
		where the constant $c_1$ and $c_2$ depend on $T$ and $(0,1)$ with $\omega\subset (0,1)$.
	\end{proposition}
	\begin{proof}[Proof of the Proposition \ref{Prop1}]
		As for every $a\in (0,\hat{a})$ we have $\beta(a)=0$ (see the assumption {\bf (H4)}), then the system \eqref{112} is given by
		\begin{equation}\label{ad11}
			\begin{cases}
				\dfrac{\partial q}{\partial t}-\dfrac{\partial q}{\partial a}-
				\dfrac{\partial q}{\partial s}-L q+(\mu_1(a)+\mu_2(s)) q=0,& \text{ in }(0,1)\times(0,\hat{a})\times(0,S)\times (0,T),\\ 
				q(x,a,s,0)=q_0(x,a,s), & \text{ in }(0,1)\times(0,\hat{a})\times(0,S).
			\end{cases}
		\end{equation}	
		If we set
		\[\tilde{q}(x,a,s,t):=q(x,a,s,t)e^{-\int\limits_{0}^{a}\mu_1(r)dr}e^{-\int\limits_{0}^{s}\mu_2(r)dr},\] 
		then $\tilde{q}$ satisfies
		\begin{align} \label{ad1}
			\dfrac{\partial \tilde{q}}{\partial t}-\dfrac{\partial \tilde{q}}{\partial a}-\dfrac{\partial \tilde{q}}{\partial s}-\Delta\tilde{q}=0\text{ in }(0,1)\times(0,\hat{a})\times(0,S)\times (0,T).
		\end{align}
		Now, let $s_*\in (0,S)$  such that $q(x,0,s,t)=0\hbox{ in }(0,1)\times(s_*,S)\times(\eta,T)$. One way to establish the inequality \eqref{Bq} is to show that there exits a constant $C>0$ such that the solution $\tilde{q}$ of \eqref{ad11} satisfies 
		
		\begin{equation}\label{cci}
			\int\limits_{\eta}^{T}\int\limits_{0}^{s_*}\int\limits_{0}^1\dfrac{1}{\sigma(x)}q^2(x,0,s,t)dxdsdt\leq C\int\limits_{0}^{T}\int\limits_{s_1}^{s_2}\int\limits_{a_1}^{a_2}\int_{0}^1\dfrac{1}{\sigma(x)}\tilde{q}^2(x,a,s,t)dxdadsdt. 
		\end{equation}
		Indeed, we have 
		\[	\int\limits_{\eta}^{T}\int\limits_{0}^{s_*}\int_{\Omega}\dfrac{1}{\sigma(x)}q^2(x,0,s,t)dxdsdt\leq e^{2\int\limits_{0}^{\hat{a}}\mu_1(r)dr+2\int\limits_{0}^{s_*}\mu_2(r)dr}\int\limits_{\eta}^{T}\int\limits_{0}^{S}\int_{\Omega}\dfrac{1}{\sigma(x)}\tilde{q}^2(x,0,s,t)dxdsdt\]
		and
		\[\int\limits_{0}^{T}\int\limits_{s_1}^{s_2}\int\limits_{a_1}^{a_2}\int_{\omega}\dfrac{1}{\sigma(x)}\tilde{q}^2(x,a,s,t)dxdadsdt\leq C(e^{2\|\mu_1\|_{L^1(0,\hat{a})}+2\|\mu_2\|_{L^1(0,S^*)}})\int\limits_{0}^{T}\int\limits_{s_1}^{s_2}\int\limits_{a_1}^{a_2}\int_{\omega}\dfrac{1}{\sigma(x)}q^2(x,a,s,t)dxdadsdt.
		\]
		As we want to estimate $q(x,0,s,t)$, we will consider the trajectory $\gamma(\lambda)=(t-\lambda,t+s-\lambda,\lambda)$ (i.e., the backward characteristics starting from $(0,s,t)$). If $T<T_1$, we can not observe all the characteristics starting from $(0,s,t)$ (see Figure 5). 
		
		Next, let $T>T_1$ and without loss the generality we assume that $\eta< T\leq \min\{a_2,\hat{a}\}$. So, to prove our claims we will proceed in two steps:
		
		\underline{{\bf Step 1 : We show that:}}
		
		$$q(\cdot,0,s,t)=0, \qquad \forall\; (s,t)\in (s_2-a_1,S)\times (a_1+S-s_2,T).$$
		First, let $s_2>a_1$. So, for every $(s,t)\in (s_2-a_1,S)\times (a_1+S-s_2,T)$ and $T_0<\min\{a_2,\hat{a}\}-a_1$, we have $$ t-A<t-S+s\hbox{ and }S-s<S-s_2+a_1<t,$$ and hence,
		\[q(x,0,s,t)=\int\limits_{t-S+s}^{t}e^{(t-l)P}\beta(t-l)q(x,0,s+t-l,l)dl.\]
		Since $0<t-l<S-s$ and $s\in (s_2-a_1,S)$, we get that $$0<t-l<S-s<S-s_2+a_1<\min\{a_2,\hat{a}\}.$$ 
		The fact that $\beta\equiv 0$ on $(0,\min\{a_2,\hat{a}\})$  yields $q(x,0,s,t)=0$. Also, it is to be noted that $$\forall\,\delta>0\hbox{ such that }s_2-a_1-\delta>0\hbox{ and }a_1+S-s_2+\delta>0$$ we have 
		$$q(\cdot,0,s,t)=0, \qquad \forall\; (s,t)\in (s_2-a_1-\delta,S)\times (a_1+S-s_2+\delta,T).$$
		In particular, if $T>S-s_2$, we have 
		$$q(\cdot,0,s,t)=0,\qquad \forall\, (s,t)\in (S-s_2,S)\times(S-s_2, T).$$
		Likewise, if $s_2<a_1,$ we have
		$$q(\cdot,0,s,t)=0,\qquad \forall\, (s,t )\in (0,S)\times (S-s_2+a_1, T).$$
		
		\underline{\textbf{Step 2: Estimation of :}} 
		
		$$q(x,0,s,t)\quad\hbox{ for }\quad(s,t)\in (0,s_2-a_1-\delta)\times (\eta,T)\quad\text{ with } \quad T_1<T,$$
		where $\delta>0$ and $\max\{s_1,a_1\}<\eta<T$.
		Note that the case $s_2<a_1$ is irrelevant because in this case \[q(x,0,s,t)=0,\qquad\hbox{ a.e. in } (0,1)\times(0,S)\times(S-s_2+a_1,T).\]
		Let $\delta>0$ be small enough, let $s\in (0,s_2-a_1-\delta)$ and $t\in (\eta,T)$. From the Figure 4, one can see that all the characteristic starting from $(0,s,t)$ passes through the observation domain $Q_1$ whenever $t>\sup\{a_1,s_1\}$ and $0\leq s<s_2-a_1$. 
		
		First, we focus on the estimate of $q(x,0,s,t)$ on $\in (0,s_2-a_1-\delta)\times\in (\eta,T)$. Two situations arise: 
		
		\underline{\textbf{Case 1: $a_1\ge s_1$.}}

		Denotes
		
		\[w(x,\lambda)=\tilde{q}(x,t-\lambda,s+t-\lambda,\lambda),\qquad (x,\lambda)\in (0,1)\times (0,t).\]
		
		Then $w$ satisfies the following PDE  
		\begin{align}
			\begin{cases}
				\dfrac{\partial w(x,\lambda)}{\partial \lambda}-L w(x,\lambda)=0,& \text{ in }  (0,1)\times (0,t),\\
				w(0,\lambda)=w(1,\lambda)=0,&  \text{ on } \times (0,t),	\\
				w(x,0)	=\tilde{q}(x,t,s+t,0), & \text{ in } (0,1).
			\end{cases}
		\end{align}
		According to Proposition \ref{observability-inequality}, for $0<t_0<t_1<t$, we have 
		
		\[\int_{0}^1\dfrac{1}{\sigma(x)}w^2(x,t)dx\leq\int_{0}^1\dfrac{1}{\sigma(x)}w^2(x,t_1)dx\leq c_1e^{\dfrac{c_2}{t_1-t_0}}\int\limits_{t_0}^{t_1}\int_{0}^1\dfrac{1}{\sigma(x)}w^2(x,\lambda)dxd\lambda.\]
		
		Or, equivalently,
		
		\[\int_{0}^1\dfrac{1}{\sigma(x)}\tilde{q}^2(x,0,s,t)dx\leq c_1e^{\dfrac{c_2}{t_1-t_0}}\int\limits_{t_0}^{t_1}\int_{\omega}\dfrac{1}{\sigma(x)}\tilde{q}^2(x,t-\lambda,s+t-\lambda,\lambda)dxd\lambda\]\[=C\int\limits_{t-t_1}^{t-t_0}\int_{\omega}\dfrac{1}{\sigma(x)}\tilde{q}^2(x,\alpha,s+\alpha,t-\alpha)dxd\alpha.\]	
		
		Then, for $t_0=t-a_1-\delta$ and $t_1=t-a_1$, we get
		
		\[\int_{0}^1\dfrac{1}{\sigma(x)}\tilde{q}^2(x,0,s,t)dx\leq C\int_{a_1}^{a_1+\delta}\int_{\omega}\dfrac{1}{\sigma(x)}\tilde{q}^2(x,\alpha,s+\alpha,t-\alpha)dxd\alpha.\]
		
		Integrating with respect to $s$ over $(0,s_2-a_1-\delta)$ we get
		
		\[\int_{0}^{s_2-a_1-\delta}\int_{0}^1\dfrac{1}{\sigma(x)}\tilde{q}^2(x,0,s,t)dxds\leq C
		\int\limits_{a_1}^{a_1+\delta}\int\limits_{a_1}^{s_2}\int_{\omega}\dfrac{1}{\sigma(x)}\tilde{q}^2(x,a,l,t-a)dxdlda.\]
		Finlay, integrating with respect to $t$ over $(\eta,T)$, we obtain
		\[\int_{\eta}^{T}\int\limits_{0}^{s_2-a_1-\delta}\int_{0}^1\dfrac{1}{\sigma(x)}\tilde{q}^2(x,0,s,t)dxdsdt\leq C \int\limits_{a_1}^{a_1+\delta}\int\limits_{s_1}^{s_2}\int\limits_{\eta-a}^{T-a}\int_{\omega}\dfrac{1}{\sigma(x)}\tilde{q}^2(x,a,s,t)dxdtdsda\]\[\leq C(\delta)\int\limits_{0}^{T} \int\limits_{a_1}^{a_2}\int\limits_{s_1}^{s_2}\int_{\omega}\dfrac{1}{\sigma(x)}\tilde{q}^2(x,a,s,t)dxdsdadt.\]
		
		Then,
		
		\begin{equation}
			\int_{\eta}^{T}\int\limits_{0}^{s_2-a_1-\delta}\int_{\Omega}\dfrac{1}{\sigma(x)}\tilde{q}^2(x,0,s,t)dxdsdt\leq C(\delta)\int\limits_{0}^{T} \int\limits_{s_1}^{s_2}\int\limits_{a_1}^{a_2}\int_{\omega}\dfrac{1}{\sigma(x)}\tilde{q}^2(x,a,s,t)dxdsdadt.
		\end{equation}
		
		\underline{\textbf{Case 2: $s_1\ge a_1$.}}
		
		In this case, we split the interval $(0,s_2-a_1-\delta)$ in two sub intervals $(0,s_1-a_1)\cup (s_1-a_1,s_2-a_1-\delta)$. Then, for $s\in (s_1-a_1,s_2-a_1-\delta)$, we consider the function $w$ defined by
		
		\[w(x,\lambda)=\tilde{q}(x,t-\lambda,s+t-\lambda,\lambda), \qquad (x,\lambda)\in (0,1)\times (0,t)).\]
		
		Then $w$ satisfies:
		\begin{equation}
			\begin{cases}
				\dfrac{\partial w(x,\lambda)}{\partial \lambda}-L w(x,\lambda)=0,&\text{ in } (0,1)\times (0,t),\\
				w(0,\lambda)=w(1,\lambda)=0,& \text{ } (0,t),	\\
				w(x,0)	=\tilde{q}(x,t,s+t,0), &\text{ in }(0,1).
			\end{cases}
		\end{equation}
		
		In view of Proposition \ref{observability-inequality}, for $0<t_0<t_1<t$, we have
		
		\[\int_{0}^1\dfrac{1}{\sigma(x)}w^2(x,t)dx\leq\int_{0}^1\dfrac{1}{\sigma(x)}w^2(x,t_1)dx\leq c_1e^{\dfrac{c_2}{t_1-t_0}}\int_{t_0}^{t_1}\int_{\omega}\dfrac{1}{\sigma(x)}w^2(x,\lambda)dxd\lambda.\]
		
		Or, equivalently,
		
		\[\int_{0}^1\dfrac{1}{\sigma(x)}\tilde{q}^2(x,0,s,t)dx\leq c_1e^{\dfrac{c_2}{t_1-t_0}}\int_{t_0}^{t_1}\int_{\omega}\dfrac{1}{\sigma(x)}\tilde{q}^2(x,t-\lambda,s+t-\lambda,\lambda)dxd\lambda\]\[=C\int_{t-t_1}^{t-t_0}\int_{\omega}\dfrac{1}{\sigma(x)}\tilde{q}^2(x,\alpha,s+\alpha,t-\alpha)dxd\alpha.\]	
		
		Thus, for $t_0=t-a_1-\delta$ and $t_1=t-a_1$, we get
		
		\[\int_{0}^1\dfrac{1}{\sigma(x)}\tilde{q}^2(x,0,s,t)dx\leq C(\delta)\int_{a_1}^{a_1+\delta}\int_{\omega}\dfrac{1}{\sigma(x)}\tilde{q}^2(x,\alpha,s+\alpha,t-\alpha)dxd\alpha.\]
		
		Integrating with respect $s$ over $(s_1-a_1,s_2-a_1-\delta)$, we get
		\[\int_{s_1-a_1}^{s_2-a_1-\delta}\int_{0}^1\dfrac{1}{\sigma(x)}\tilde{q}^2(x,0,s,t)dxds\leq C(\delta)
		\int_{a_1}^{a_1+\delta}\int\limits_{s_1}^{s_2}\int_{\omega}\dfrac{1}{\sigma(x)}\tilde{q}^2(x,a,l,t-a)dxdlda.\]
		Finally, integrating with respect $t$ over $(\eta,T)$, we obtain
		\[\int\limits_{\eta}^{T}\int\limits_{s_1-a_1}^{s_2-s_1-\delta}\int_{0}^1\dfrac{1}{\sigma(x)}\tilde{q}^2(x,0,s,t)dxdsdt\leq C(\delta) \int\limits_{a_1}^{a_1+\delta}\int_{s_1}^{s_2}\int\limits_{\eta-a}^{T-a}\int_{\omega}\dfrac{1}{\sigma(x)}\tilde{q}^2(x,a,s,t)dxdtdsda\]\[\leq C(\delta)\int_{0}^{T} \int_{a_1}^{a_2}\int_{s_1}^{s_2}\int_{\omega}\dfrac{1}{\sigma(x)}\tilde{q}^2(x,a,s,t)dxdsdadt.\]
		Therefore,
		\begin{equation}
			\int\limits_{\eta}^{T}\int\limits_{s_1-a_1}^{s_2-a_1-\delta}\int_{\Omega}\dfrac{1}{\sigma(x)}\tilde{q}^2(x,0,s,t)dxdsdt\leq C(\delta)\int\limits_{0}^{T} \int\limits_{s_1}^{s_2}\int\limits_{a_1}^{a_2}\int_{\omega}\dfrac{1}{\sigma(x)}\tilde{q}^2(x,a,s,t)dxdsdadt.\label{rr}
		\end{equation}
		Now, we estimate $q(x,0,s,t)$ on $(0,s_1-a_1)\times (\eta,T) $. Here, too, two situations arise,
		\[s_2-s_1> s_1-a_1\hbox{ and }s_2-s_1<s_1-a_1.\]
		If $s_2-s_1<s_1-a_1,$ we split $(0,s_1-a_1)$ as follows
		\[(0,s_2-s_1)\cup(s_2-s_1,s_1-a_1).\]
		Next, the calculations are done only for the case $s_2-s_1>s_1-a_1$ as for the second case it follows in the same way. If we set 
		\[w(x,\lambda)=\tilde{q}(x,t-\lambda,s+t-\lambda,\lambda), \qquad (x,\lambda)\in (0,1)\times (0,t),\]
		then $w$ satisfies:
		\begin{align}
			\begin{cases}
				\dfrac{\partial w(x,\lambda)}{\partial \lambda}-L w(x,\lambda)=0, &\text{ in } (0,1)\times (0,t),\\
				w(0,\lambda)=w(1,\lambda)=0, & \text{ on } \times (0,t),	\\
				w(x,0)	=\tilde{q}(x,t,s+t,0), &\text{ in }(0,1).
			\end{cases}
		\end{align}
		Therefore, according to Proposition \ref{observability-inequality}, we have
		\[\int_{0}^1\dfrac{1}{\sigma(x)}w^2(x,t)dx\leq\int_{0}^1\dfrac{1}{\sigma(x)}w^2(x,t_1)dx\leq c_1e^{\dfrac{c_2}{t_1-t_0}}\int\limits_{t_0}^{t_1}\int_{\omega}\dfrac{1}{\sigma(x)}w^2(x,\lambda)dxd\lambda,\]
		for $0<t_0<t_1<t$. Equivalently,
		\[\int_{0}^1\dfrac{1}{\sigma(x)}\tilde{q}^2(x,0,s,t)dx\leq c_1e^{\dfrac{c_2}{t_1-t_0}}\int\limits_{t_0}^{t_1}\int_{\omega}\dfrac{1}{\sigma(x)}\tilde{q}^2(x,t-\lambda,s+t-\lambda,\lambda)dxd\lambda\]\[=C\int\limits_{t-t_1}^{t-t_0}\int_{\omega}\dfrac{1}{\sigma(x)}\tilde{q}^2(x,\alpha,s+\alpha,t-\alpha)dxd\alpha.\]	
		Then, for $t_0=t-s_1-\kappa$ with $\kappa>0$ and $t_1=t-s_1$, we get 
		\[\int_{0}^1\dfrac{1}{\sigma(x)}\tilde{q}^2(x,0,s,t)dx\leq C\int\limits_{s_1}^{s_1+\kappa}\int_{\omega}\dfrac{1}{\sigma(x)}\tilde{q}^2(x,l-s,l-s,t-l+s)dxd\alpha.\]
		Integrating with respect to $s$ over $(a_0,s_1-a_1)$, we get
		\[\int\limits_{a_0}^{s_1-a_1}\int_{\Omega}\dfrac{1}{\sigma(x)}\tilde{q}^2(x,0,s,t)dxds\leq C
		\int\limits_{s_1}^{s_1+\kappa}\int\limits_{l-s_1+a_1}^{l-a_0}\int_{\omega}\dfrac{1}{\sigma(x)}\tilde{q}^2(x,a,l,t-a)dxdlda.\] Then
		\[\int\limits_{a_0}^{s_1-a_1}\int_{\Omega}\dfrac{1}{\sigma(x)}\tilde{q}^2(x,0,s,t)dxds\leq C
		\int\limits_{s_1}^{s_1+\kappa}\int\limits_{a_1}^{s_1+\kappa-a_0}\int_{\omega}\dfrac{1}{\sigma(x)}\tilde{q}^2(x,a,l,t-a)dxdlda.\]
		Finlay, integrating with respect to $t$ over $(\eta,T)$, we obtain
		\[\int\limits_{\eta}^{T}\int\limits_{a_0}^{s_1-a_1}\int_{\Omega}\dfrac{1}{\sigma(x)}\tilde{q}^2(x,0,s,t)dxdsdt\leq C \int\limits_{s_1}^{s_1+\kappa}\int\limits_{a_1}^{s_1+\kappa-a_0}\int\limits_{\eta-a}^{T-a}\int_{\omega}\dfrac{1}{\sigma(x)}\tilde{q}^2(x,a,s,t)dxdtdsda,\]
		where $\kappa>0$ is choose small enough such that \[s_1-a_0+\kappa<T<\min\{a_2,\hat{a}\}.\]
		It follows that
		\begin{equation}
			\int_{\eta}^{T}\int\limits_{a_0}^{s_1-a_1}\int_{0}^1\dfrac{1}{\sigma(x)}\tilde{q}^2(x,0,s,t)dxdsdt\leq C(\kappa)\int\limits_{0}^{T} \int\limits_{s_1}^{s_2}\int\limits_{a_1}^{a_2}\int_{\omega}\dfrac{1}{\sigma(x)}\tilde{q}^2(x,a,s,t)dxdsdadt.
		\end{equation}
		Finally, \eqref{rr} together with the fact that
		$$q(x,0,s,t)=0\hbox{ for } t\in (S-s_2+a_1+\delta,+\infty)\hbox{ } s\in (s_2-a_1-\delta,S)$$ implies 
		\[\int\limits_{\eta}^{T}\int\limits_{0}^{S}\int_{0}^1\dfrac{1}{\sigma(x)}\tilde{q}^2(x,0,s,t)dxdsdt\leq C\int\limits_{0}^{T} \int_{s_1}^{s_2}\int\limits_{a_1}^{a_2}\int_{\omega}\dfrac{1}{\sigma(x)}\tilde{q}^2(x,a,s,t)dxdsdadt\] where $\max\{S-s_2+a_1,s_1\}<\eta<T.$
		\begin{remark}
			It is to be noted that in all the above cases, when $\delta \to 0$, or $ \kappa \to 0$, or $\eta \to  T_1$ we get $C(\delta,\kappa,\eta)\to  +\infty.$ 
		\end{remark}
		\begin{remark}
			For the case $\eta<\min\{\hat{a},a_2\}\leq T,$ we split the interval $(\eta,T)$ in two sub-intervals $$(\eta,\min\{\hat{a},a_2\})\cup  (\min\{\hat{a},a_2\},T)$$.In $(\eta,\min\{\hat{a},a_2\}),$ we proceed in the same way. In $(\min\{\hat{a},a_2\},T)$ we proceed as in the proof of \cite[Proposition 3.7]{abst}.
		\end{remark}
		
	\end{proof}
	\begin{figure}[H]\label{dada}
		%		\begin{subfigure}{0.4\textwidth
		\begin{overpic}[scale=0.6]{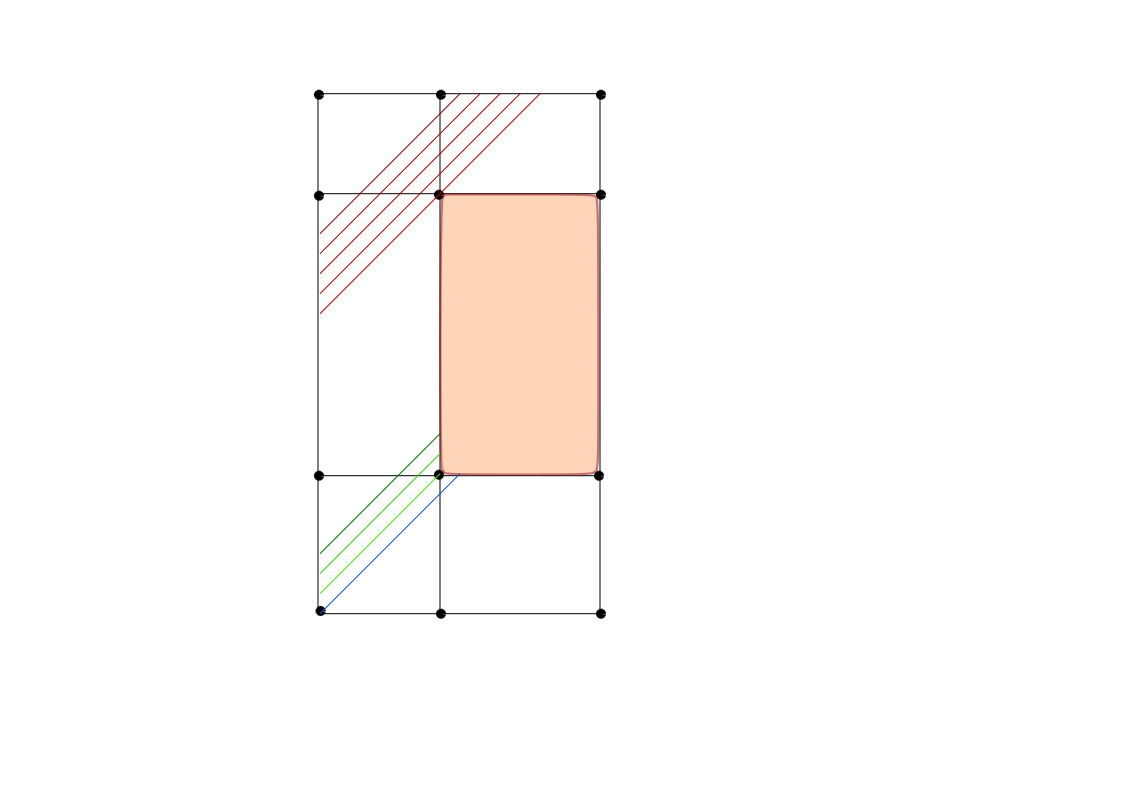}
			\put (25,62) {$S$}	
			\put (53.5,14) {$\hat{a}$}
			\put (25,28) {$s_1$}
			\put (25,53) {$s_2$}				
			\put (38.7,14) {$a_1$}
			\put (42.5,35.5) {$Q_{1}$}		
		\end{overpic}
		\caption{An illustration of the estimation of $q(x,0,s,t)$. Here we have chosen $a_2=\hat{a}$. Since $t>T_1$ all backward characteristics starting from $(0,s,t)$ enter the observation domain (the green and blue lines), or leave the domain $\Sigma$ through the boundary $s=S$ (or at $t=0$) (the red line).}
		%\end{subfigure}
	\end{figure}
	\begin{figure}[H]
		\begin{subfigure}{0.5\textwidth}
			\begin{overpic}[scale=0.3]{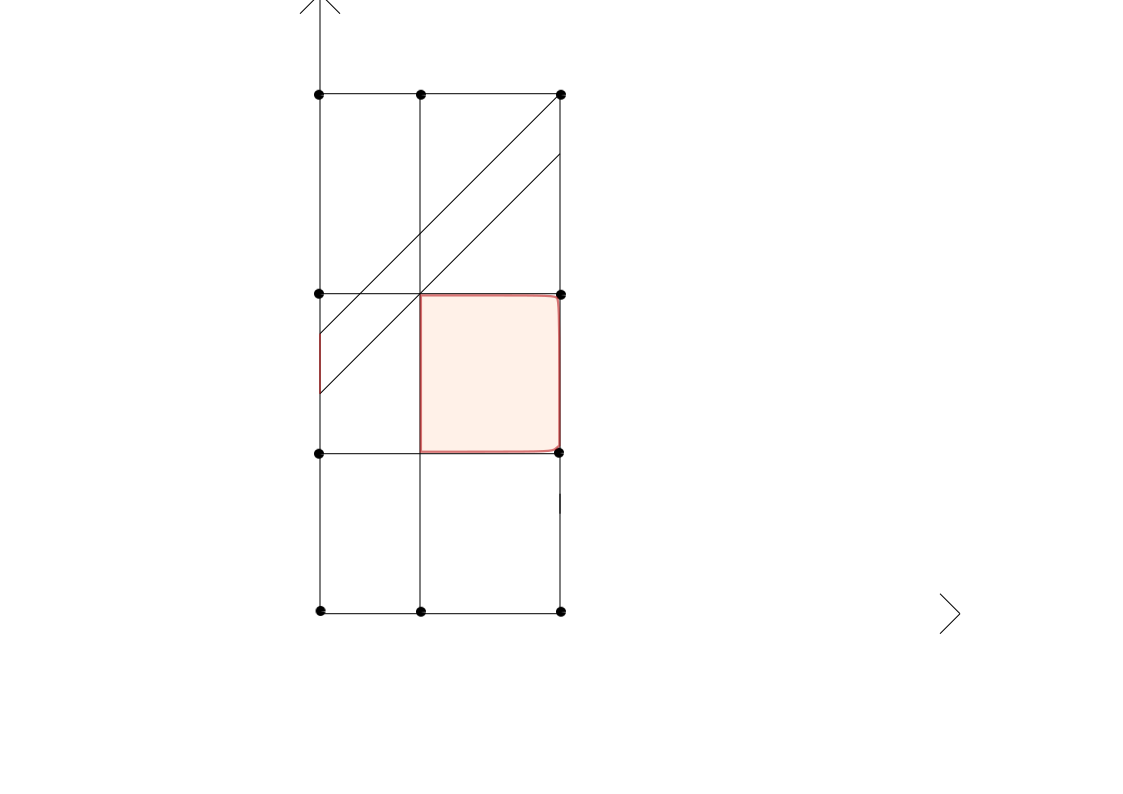}
				\put (25,62) {$S$}	
				\put (69.5,12) {$A$}
				\put (49.5,12.5) {$\hat{a}$}
				\put (25,30) {$s_1$}
				\put (25,44) {$s_2$}
				\put (25.5,35) {$c_1$}
				\put (25.5,41) {$c_2$}				
				\put (36.5,13) {$a_1$}
				\put (42.5,35.5) {$Q_{1}$}		
			\end{overpic}
			\subcaption{For $S-s_2>\hat{a}-a_1,$ we can not estimate $q (x,0,s,t)$ for $s\in(c_1,c_2)$ by the characteristic method. Indeed for even $t>a_1+S-s_2$ the characteristics starting at $(0,s,t)$ leave the domain $\Sigma$ at $t=0, \hbox{ or enter in the region } a>\hat{a},$ without going through the observation domain.}
		\end{subfigure}\quad
		\begin{subfigure}{0.5\textwidth}
			\begin{overpic}[scale=0.3]{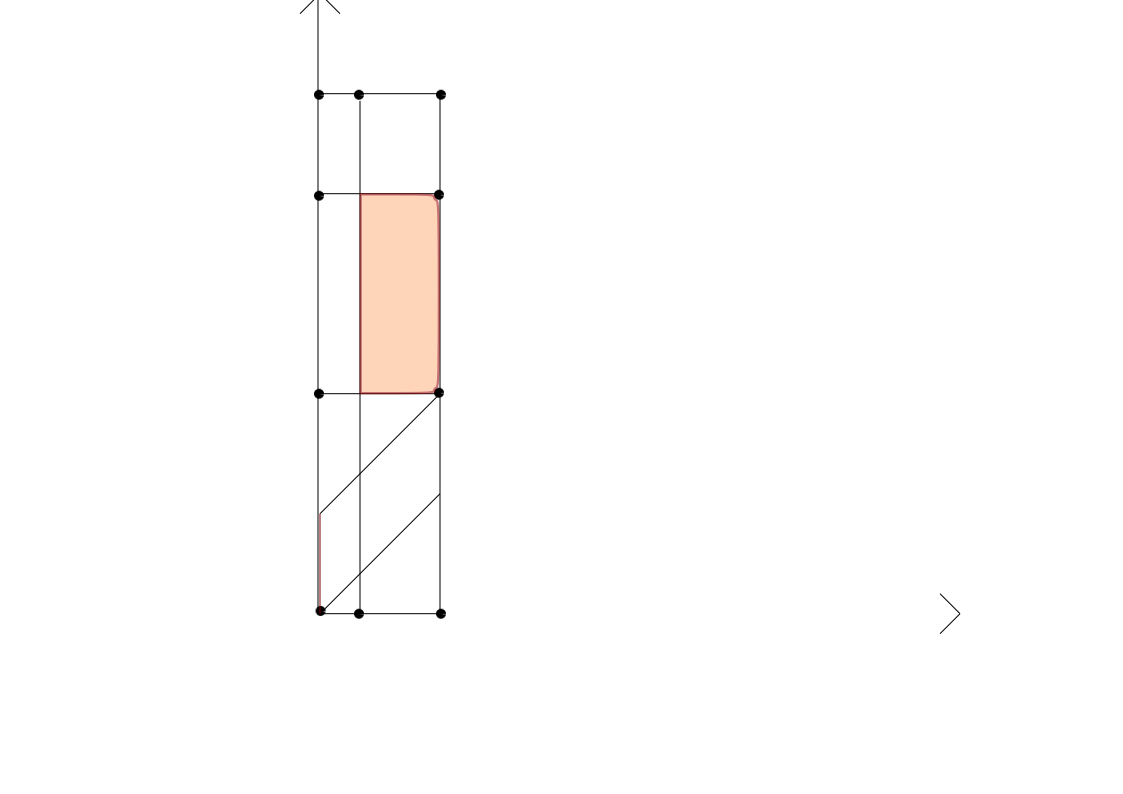}
				\put (25,62) {$S$}	
				\put (38,12.5) {$\hat{a}$}
				\put (25,35) {$s_1$}
				\put (25,53) {$s_2$}				
				\put (31.5,13) {$a_1$}
				\put (34.5,44.5) {$Q_{1}$}	
				\put (25.5,16) {$c_1$}
				\put (25.5,24) {$c_2$}
			\end{overpic}
			\subcaption{For the second case if $s_1>\hat{a}-a_1,$ we can not estimate $q (x,0,s,t)$ for $s\in (c_1,c_2)$ by the characteristic method. Indeed, for even $t>a_1+S-s_2$ the characteristics starting at $(0,s,t)$ leave the domain $\Sigma$ at the boundary $t=0, \hbox{ or enter in the region} a>\hat{a}$, without going through the observation domain.}
		\end{subfigure}
		\caption{Illustration of cases where we can not estimate the non-local term $q(x,0,s,t)$}
	\end{figure}

	\begin{proposition}\label{p2}
		Let the assumptions of Theorem \ref{Main2} be satisfied and let $T>a_1+T_0$. Then for every $q_{0}\in K$, the solution $q$ of the system \eqref{112} obeys
		\begin{equation}
			\displaystyle\int\limits_{0}^{S}\int\limits_{0}^{a_1}\int\limits_{0}^1\dfrac{1}{\sigma(x)}q^2(x,a,s,T)dxdads\leq C_T\displaystyle\int\limits_{0}^{T}\int\limits_{s_1}^{s_2}\int\limits_{a_1}^{a_2}\int\limits_{\l_1}^{l_2}\dfrac{1}{\sigma(x)}q^2(x,a,s,t)dxdadsdt \label{C81}
		\end{equation}
	\end{proposition}
	
	\begin{proof}[Proof of the Proposition \ref{p2} ]
		Let us consider the operator $\mathbf{A}: D(\mathbf{A}) \to L^{2}_{\sigma}((0,1)\times (0,S)) $ defined by $$\mathbf{A}\psi:=\partial_{s}\psi-L\psi+\mu_2(s)\psi,\quad D(\mathbf{A}):=\left\{\varphi/ \mathbf{A}\varphi\in L^{2}_{\sigma}((0,1)\times (0,S)), \; \varphi(x,0)=0,\;  \varphi(0,\cdot)=\varphi(1,\cdot)=0\right\}.$$ 
		Then, the operator $\mathcal{A}$ can be rewritten as \[\mathcal{A}\varphi:=-\partial_{a}\varphi-\mathbf{A}\varphi-\mu(a)\varphi, \qquad \varphi\in D(\mathcal{A}),\]
		and hence the adjoint operator $\mathcal{A}^*$ of $\mathcal{A}$ is given by
		\[\mathcal{A}^*\varphi:=\partial_{a}\varphi-\mathbf{A}^*\varphi-\mu(a)\varphi +\beta(a)\varphi(\cdot,0,\cdot,\cdot),\qquad \varphi\in D(\mathcal{A^*}).\]
		Moreover, we select the following operators $$\mathbf{B}:=\mathbf{1}_{(s_1,s_2)}\mathbf{1}_{\omega}\qquad\hbox{and} \qquad \mathcal{B}:=\mathbf{1}_{(a_1,a_2)}\mathbf{B}.$$
		Note that the operator $(\mathbf{A}^*,\mathbf{B}^*)$ is final state observable for every $T>T_0$, cf. \cite{b1}. Therefore, \cite[Proposition 3.5]{abst} yields the inequality \eqref{C81}.
	\end{proof}
	
	\begin{remark}
		Note that the result of Proposition \ref{p2} holds for $T>T_1$. However, it does not improve the main result by using the approach developed in \cite{b1}.
	\end{remark}
	
	Now, let us consider the following cascade system
	\begin{equation}\label{322}
		\left\lbrace
		\begin{array}{ll}
			\dfrac{\partial w}{\partial t}-\dfrac{\partial w}{\partial a}-\dfrac{\partial w}{\partial s}-L w+(\mu_1(a)+\mu_2(s))w=0, &\hbox{ in }Q ,\\ 
			w(0,a,s,t)=w(1,a,s,t)=0, &\hbox{ on }\Sigma,\\ 
			w\left( x,A,s,t\right) =0, &\hbox{ in } Q_{S,T}, \\
			w(x,a,S,t)=0, &\hbox{ in } Q_{S,T},\\
			w\left(x,a,s,0\right)=w_0,&\hbox{ in }Q_{A,S}.
		\end{array}\right.
	\end{equation} 
	We also need the following result for the proof of Theorem \ref{Main3}.
	\begin{proposition}\label{1p}
		Let the assumption of Theorem \ref{Main2} be satisfied. Let $s_0\in(s_1,s_2)$, $T>s_1$ and $a_1<a_0<a_2-s_1.$ Then, there exists a constant $C_T>0$ such that the solution $w$ of the system \eqref{322} satisfies the following inequality
		\begin{equation}\label{C8}
			\displaystyle\int\limits_{0}^{s_0}\int\limits_{a_1}^{a_0}\int\limits_{0}^1\dfrac{1}{\sigma(x)}w^2(x,a,s,T)dxdads\leq C_T\displaystyle\int\limits_{0}^{T}\int\limits_{s_1}^{s_2}\int\limits_{a_1}^{a_2}\int_{\omega}\dfrac{1}{\sigma(x)}w^2(x,a,s,t)dxdadsdt.
		\end{equation}
		Moreover, if $0<\epsilon<\delta$ is chosen such that $T>S-s_2+\delta,$ and $s_0=s_2-\epsilon,$ then
		\begin{align}\label{C9}
			w(x,a,s,T)=0,\quad\hbox{ a.e. in }(0,1)\times (a_1,A)\times (s_0,S).
		\end{align}
	\end{proposition}
	
	\begin{proof}[Proof of Proposition \ref{1p}]
		
		The proof is provided in two parts:
		
		\underline{\textbf{1)- Proof of the inequality $\eqref{C8}$:}}
		
		Let $(\textbf{A}^*,\textbf{B}^*)$ as defined in the proof of Proposition \ref{p2}. Then, the operator $(\mathbf{A}^*,\mathbf{B}^*)$ is final state observable for any $T>T_0=\max\{s_1,S-s_2\}$, cf. \cite{b1}. In particular, for every $T>s_1$ and $s_0\in (s_1,s_2)$, there exists a constant $C_T>0$ such that
		
		\begin{equation}\label{1xx2}
			\int\limits_{0}^{s_0}\int\limits_{0}^{1}\dfrac{1}{\sigma(x)}\varphi^2(x,s,T)dxds  \leq C_T \int\limits_{0}^{T}\int\limits_{s_1}^{s_2}\int\limits_{l_1}^{l_2}\dfrac{1}{\sigma(x)}\varphi^2(x,s,t)dxdsdt,
		\end{equation} 
		where $\varphi$ satisfies 
		\begin{equation}\label{xx}
			\begin{cases}
				\dfrac{\partial \varphi}{\partial t}-\dfrac{\partial \varphi}{\partial s}-L \varphi+\mu_2(s)\varphi=0, &\hbox{ in } (0,1)\times(0,S)\times (0,T) ,\\ 
				\varphi(0,s,t)=\varphi(1,s,t)=0, &\hbox{ on }(0,S)\times (0,T),\\ 
				\varphi(x,S,t)=0, &\hbox{ in } (0,1)\times (0,T),\\
				\varphi\left(x,s,0\right)=\varphi_0, &\hbox{ in } (0,1)\times (0,S).
			\end{cases}
		\end{equation} 
		Then, we have the following result. 
		\begin{lemma}\label{p1}
			Let the assumption of Theorem \ref{Main2} be satisfied. Let $T>s_1$ and $C_T$ be the observability cost in \eqref{1xx2} which satisfies $$C_T\to +\infty, \hbox{ as }T\mapsto s_1.$$
			Furthermore, let $T_1$, $T_2$ and $T_3$ be real numbers such that \[0\leq T_1< T_2\leq T_3, \quad\hbox{ with } \quad T_2-T_1>s_1.\] Then, for every $\varphi_0\in D(\textbf{A}^*)$, the solution $\varphi$ of the problem 
			\begin{align*}
				\dfrac{d \varphi}{dt}=\textbf{A}^* \varphi, \qquad t\in [T_1,T_3],\qquad \varphi(T_1)=\varphi_0,
			\end{align*}
			satisfies the estimate 
			\begin{align}\label{EstimateA}
				\|\varphi(T_3)\|_{L_{\frac{1}{\sigma}}^{2}((0,1)\times(0,s_0))}^2\leq Me^{c_0(T_2-T_1)}C(T_2-T_1)\int\limits_{T_1}^{T_2}\|\varphi(t)\|_{L^{2}_{\frac{1}{\sigma}}((l_1,l_2)\times(s_1,s_2))}dt.
			\end{align}
		\end{lemma}
		
		\begin{proof}[Proof of Lemma \ref{p1}]
			Let $\textbf{C}: D(\textbf{C})\to L^2_{\sigma}((0,1)\times (0,s_0))$ the operator defined by
			\begin{align*}
				\textbf{C}v & =\dfrac{\partial v}{\partial s}+L v-\mu_2(s)v,\cr
				D(\textbf{C}) &=\left\{v\in L^2_{\sigma}((0,1)\times (0,s_0)):\; 	\textbf{C}v\in L^2_{\sigma}((0,1)\times (0,s_0)),\; v(0,\cdot)=v(1,\cdot)=0,\; v(.,s_0)=0\right\}.
			\end{align*}
			It is not difficult to prove that the operator $(\textbf{C},D(\textbf{C}))$ generates a C$_0$-semigroup on $L^2_{\sigma}((0,1)\times (0,s_0))$. Thus, for every $T_2\leq T_3$, there exists $C_{T_{2,3}}>0$ such that \[\int\limits_{0}^{s_0}\int\limits_{0}^{1}\dfrac{1}{\sigma(x)}\varphi^2(x,s,T_3)dxds\leq C_{T_{2,3}}\int\limits_{0}^{S_0}\int\limits_{0}^{1}\dfrac{1}{\sigma(x)}\varphi^2(x,s,T_2)dxds.\]
			Now applying the estimate $\eqref{1xx2}$ on the time interval $[T_1,T_2]$ we get
			\[\int\limits_{0}^{s_0}\int\limits_{0}^{1}\dfrac{1}{\sigma(x)}\varphi^2(x,s,T_2)dxds\leq C_{T_{1,2}}\int\limits_{T_1}^{T_2}\int\limits_{s_1}^{s_2}\int\limits_{l_1}^{l_2}\dfrac{1}{\sigma(x)}\varphi^2(x,s,t)dxdsdt,\]
			for a constant $C_{T_{1,2}}>0$. The combination of the two estimates above gives \eqref{EstimateA}.
		\end{proof}
		
		In order to apply the result of the above lemma, we need to rewrite \eqref{322} as follows
		\begin{align}\label{ad}
			\begin{cases}
				\dfrac{\partial w}{\partial t}-\dfrac{\partial w}{\partial a}-\textbf{A}^*w+\mu_1(a)w=0, &\hbox{ in } (0,1)\times(0,A)\times(0,T),\\ 
				w(x,A,t)=0,  &\hbox{ on }(0,1)\times (0,T),\\
				w(x,a,0)=w_0,& \hbox{ in } (0,1)\times (0,A) .
			\end{cases}
		\end{align}	
		where $w_0\in L^2(0,A;L^2((0,1)\times(0,S)))$.  Moreover, if we set  $$\tilde{w}(x,a,s,t)=w(x,a,s,t)\exp\left(-\int_{0}^{a}\mu_1(r)dr\right),$$ 
		then $\tilde{w}$ satisfies
		\begin{equation}\label{ad1}
			\begin{cases}
				\dfrac{\partial \tilde{w}}{\partial t}-\dfrac{\partial \tilde{w}}{\partial a}-\textbf{A}^*\tilde{w}=0, &\hbox{ in } (0,1)\times(0,A)\times(0,T),\\ 
				\tilde{w}(x,A,t)=0,  &\hbox{ on }(0,1)\times (0,T),\\
				\tilde{w}(x,a,0)=\tilde{w}_0,& \hbox{ in } (0,1)\times (0,A) .
			\end{cases}
		\end{equation}
		
		As before, proving the inequality \eqref{C8} is equivalent to shows that there exits a constant $C>0$ such that the solution $\tilde{w}$ of \eqref{ad1} satisfies 
		
		\begin{equation}\label{cci}
			\int\limits_{0}^{A}\int\limits_{0}^{s_0}\int\limits_{0}^1\dfrac{1}{\sigma(x)}\tilde{w}^2(x,a,s,T)dxdsda\leq C\int\limits_{0}^{T}\int\limits_{a_1}^{a_2}\int\limits_{s_1}^{s_2}\int\limits_{l_1}^{l_2}\dfrac{1}{\sigma(x)}\tilde{w}^2(x,a,s,t)dxdsdadt. 
		\end{equation}
		
		Now, let us consider the trajectory $\gamma(\lambda)=(a+T-\lambda,\lambda)$ (i.e., the backward characteristics starting from $(a,0)$) and set
		\begin{align*}
			\varphi(x,s,\lambda)=\tilde{w}(x,a+T-\lambda,s,\lambda).
		\end{align*}
		We clearly have 
		\[\dfrac{d\varphi}{d\lambda}=\textbf{A}^*\varphi,\qquad \hbox{ in }(0,T).\]
		Let us assume, without loss of generality, that $T<a_2$ and $ T>a_2-a_1$. As such, we split the interval $(a_1,a_0)$ as $$(a_1,a_0)=(a_1,a_3)\cup(a_3,a_0),\quad\hbox{ with }\quad a_3=a_2-T.$$
		
		\underline{\textbf{Upper bound on $(a_1,a_3)$:}}
		
		According to Lemma \ref{p1}, there exists a constant $C_1>0$ such that  
		\begin{align*}
			\|\varphi(T)\|^{2}_{L^{2}_{\frac{1}{\sigma}}((0,1)\times(0,s_0))}\leq C_1\int_{T_1}^{T_2}\|\varphi(\lambda)\|^{2}_{L^{2}_{\frac{1}{\sigma}}(\omega\times(s_,s_2))}
		\end{align*}
		Equivalently, 
		\begin{align*}
			\int_{0}^{s_0}\int_{0}^1\dfrac{1}{\sigma(x)}\tilde{w}^2(x,a,s,T)dxda\leq C_1\int_{T_1}^{T_2}\int_{s_1}^{s_2}\int_{\omega}\dfrac{1}{\sigma(x)}\tilde{w}^2(x,a+T-\lambda,s,\lambda)dxdsd\lambda.
		\end{align*}
		So, by selecting \[T_1=0\quad\hbox{ and } \quad T_2=T+a-a_1,\] we get
		\[\int_{0}^{s_0}\int_{0}^1\dfrac{1}{\sigma(x)}\tilde{w}^2(x,a,s,T)dxda\leq C_1\int_{a_1}^{a+T}\int_{s_1}^{s_2}\int_{\omega}\dfrac{1}{\sigma(x)}\tilde{w}^2(x,l,s,a+T-l)dxdsdl.\]
		Therefore, integrating with respect to $a$ over $(a_1,a_3),$ we obtain
		\[\int_{a_1}^{a_3}\int_{0}^{s_0}\int_{0}^1\dfrac{1}{\sigma(x)}\tilde{w}^2(x,a,s,T)dxda\leq C_1\int_{a_1}^{a_2}\int_{s_1}^{s_2}\int_{\omega}\dfrac{1}{\sigma(x)}\int_{l-T}^{a_3}\tilde{w}^2(x,l,s,a+T-l)dadxdsdl.\]
		Thus, 
		\[\int_{a_1}^{a_3}\int_{0}^{s_0}\int_{0}^1\dfrac{1}{\sigma(x)}\tilde{w}^2(x,a,s,T)dxda\leq C_1\int_{a_1}^{a_2}\int_{s_1}^{s_2}\int_{\omega}\int\limits_{0}^{a_2-l}\dfrac{1}{\sigma(x)}\tilde{w}^2(x,l,s,t)dtdxdsdlda.\]
		Hence,
		\[\int_{a_1}^{a_3}\int_{0}^{s_0}\int_{0}^1\dfrac{1}{\sigma(x)}\tilde{w}^2(x,a,s,T)dxda\leq C_1\int_{0}^{T}\int_{a_1}^{a_2}\int\limits_{s_1}^{s_2}\int_{\omega}\dfrac{1}{\sigma(x)}\tilde{w}^2(x,a,s,t)dxdsdldadt.\]
		
		\underline{\textbf{Upper bound on $(a_3,a_0)$:}}
		
		Similarly, by virtue of Lemma \ref{p1}, we have 
		\[\|\varphi(T)\|^{2}_{L^2((0,1)\times(0,s_0))}\leq C_1\int_{T_1}^{T_2}\|\varphi(\lambda)\|^{2}_{L^2(\omega\times(s_1,s_2))},\]
		for a constant $C_1>0$. Or, equivalently, 
		\[\int_{0}^{s_0}\int_{0}^1\dfrac{1}{\sigma(x)}\tilde{w}^2(x,a,s,T)dxda\leq C_1\int_{T_1}^{T_2}\int_{s_1}^{s_2}\int_{\omega}\dfrac{1}{\sigma(x)}\tilde{w}^2(x,a+T-\lambda,s,\lambda)dxdsd\lambda.\]
		So, by selecting \[T_1=T+a-a_2\hbox{ and }T_2=T+a-a_1,\] we get
		\[\int_{0}^{s_0}\int_{0}^1\dfrac{1}{\sigma(x)}\tilde{w}^2(x,a,s,T)dxda\leq C_1\int_{a_1}^{a_2}\int_{s_1}^{s_2}\int_{\omega}\dfrac{1}{\sigma(x)}\tilde{w}^2(x,l,s,a+T-l)dxdsdl.\]
		Integrating with respect to $a$ over $(a_3,a_2)$, we get
		\[\int_{a_3}^{a_2}\int_{0}^{s_0}\int_{0}^1\dfrac{1}{\sigma(x)}\tilde{w}^2(x,a,s,T)dxda\leq C_1\int\limits_{a_1}^{a_2}\int_{s_1}^{s_2}\int_{\omega}\int_{a_3}^{a_2}\dfrac{1}{\sigma(x)}\tilde{w}^2(x,l,s,a+T-l)dadxdsdl.\]
		Thus,
		\[\int_{a_3}^{a_2}\int_{0}^{s_0}\int_{0}^1\dfrac{1}{\sigma(x)}\tilde{w}^2(x,a,s,T)dxda\leq C_1\int_{a_1}^{a_2}\int\limits_{s_1}^{s_2}\int_{\omega}\int\limits_{a_3+T-l}^{a_2+T-l}\dfrac{1}{\sigma(x)}\tilde{w}^2(x,l,s,t)dtdxdsdlda.\]
		Hence, 
		\[\int_{a_1}^{a_3}\int_{0}^{s_0}\int_{0}^1\dfrac{1}{\sigma(x)}\tilde{w}^2(x,a,s,T)dxda\leq C_1\int_{0}^{T}\int_{a_1}^{a_2}\int_{s_1}^{s_2}\int_{\omega}\dfrac{1}{\sigma(x)}\tilde{w}^2(x,a,s,t)dxdsdldadt.\]
		Note that if $T>a_2$ or if $T<a_2$ and $T\geq a_2-a_1 $, we estimate directly on the interval $(a_1,a_2).$
		
		\begin{figure}[H]1
			\begin{subfigure}{0.4\textwidth}
				\begin{overpic}[scale=0.35]{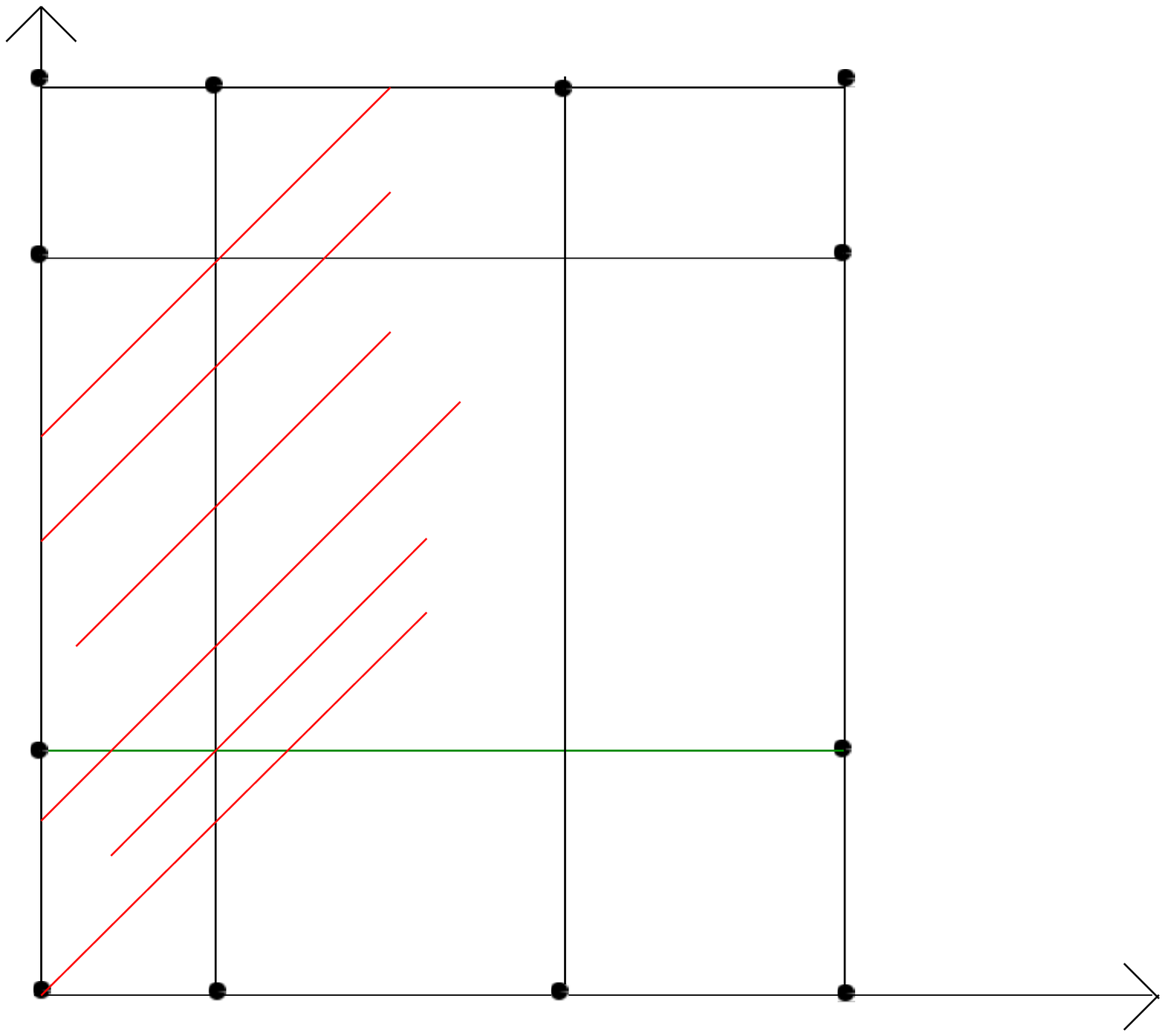}
					\put (25,60.5) {$S$}
					\put (25,52.5) {$s_2$}
					\put (40,34.5) {$Q_1$}	
					\put (25,27) {$s_1$}
					\put (36,12.5) {$a_1$}
					\put (53,12.5) {$a_2$}
					\put (67,11.5) {$A$}
					\put(28.5,24.1){\makebox[0pt]{\Huge\textcolor{black}{.}}}
					\put(32.3,22.5){\makebox[0pt]{\Huge\textcolor{black}{.}}}
					\put(30.2,32.85){\makebox[0pt]{\Huge\textcolor{black}{.}}}
					\put(28.5,43.8){\makebox[0pt]{\Huge\textcolor{black}{.}}}
					\put(28.5,38.3){\makebox[0pt]{\Huge\textcolor{black}{.}}}
				\end{overpic}
				\caption{Estimation of $q(x,a,s,T)$ on $(0,a_1)$. For $T>T_0+a_1$, the backward characteristics starting from $(a,s,T)$ with $a\in (0,a_1)$ enter the observation domain or without and leave the domain $\Sigma$ by the boundary $s=S$. }
			\end{subfigure}\quad
			\begin{subfigure}{0.4\textwidth}
				\begin{overpic}[scale=0.35]{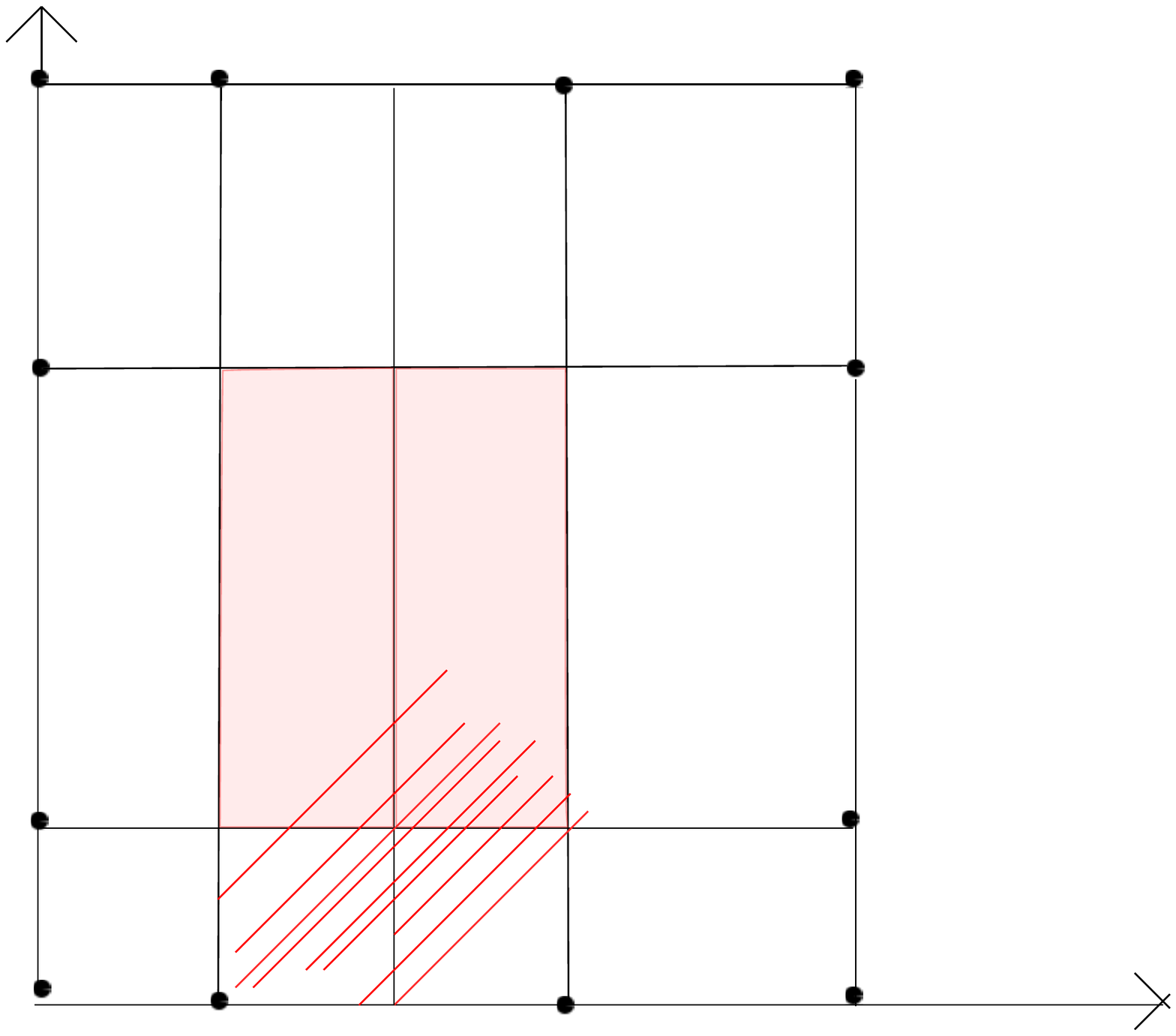}
					\put (25,60.5) {$S$}	
					\put (25,47) {$s_2$}
					\put (25,23) {$s_1$}
					\put (45,13) {$c$}
					\put (42,32.5) {$Q_1$}
					\put (36,12.5) {$a_1$}
					\put (54,12.5) {$a_2$}
					\put (67,11.5) {$A$}
					\put (42,8) {$c=a_2-s_1$}
					\put(38.6,18){\makebox[0pt]{\Huge\textcolor{black}{.}}}
					\put(39.2,16){\makebox[0pt]{\Huge\textcolor{black}{.}}}
					\put(46.4,18.9){\makebox[0pt]{\Huge\textcolor{black}{.}}}
					\put(46.4,15.2){\makebox[0pt]{\Huge\textcolor{black}{.}}}
					\put(44.8,15.2){\makebox[0pt]{\Huge\textcolor{black}{.}}}
					\put(37.5,20.4){\makebox[0pt]{\Huge\textcolor{black}{.}}}
					\put(42,16.9){\makebox[0pt]{\Huge\textcolor{black}{.}}}
					\put(42.8,16.9){\makebox[0pt]{\Huge\textcolor{black}{.}}}
					\put(38.4,16.2){\makebox[0pt]{\Huge\textcolor{black}{.}}}	
				\end{overpic}
				\caption{Estimation of $q(x,a,s,T)$ on $(a_1,a_0)$. For $T>s_1,$ the backward characteristics starting from $(a,s,T)$ with $a\in (a_1,a_2-s_1)$ and $s\in (0,s_0)$ enter the observation domain.
					%for $s\in(0,s_2-\epsilon)$ or without and leave the domain $\Sigma$ by the boundary $s=S$ for $s\in (s_2-\epsilon,S)$
				}
			\end{subfigure}
			\caption{Illustration of the time to estimate $q(x,a,s,T)$ on $(0,a_0)$ }
		\end{figure}
		
		\underline{\textbf{2)- Proof of the equality \eqref{C9}:}}
		
		To this end, let $v$ a function so that 
		\begin{equation}\label{xx1}
			\begin{cases}
				\dfrac{\partial v}{\partial t}-\dfrac{\partial v}{\partial a}-\dfrac{\partial v}{\partial s}-L u+(\mu_1(a)+\mu_2(s))v=0,&\hbox{ in }(0,1)\times (0,A)\times(s_0,S)\times (\eta,T) ,\\ 
				v(0,a,s,t)=v(1,a,s,t)=0,&\hbox{ on }(0,A)\times(s_0,S)\times (\eta,T),\\ 
				v\left( x,A,s,t\right) =0,&\hbox{ in } (0,1)\times(s_0,S)\times (\eta,T) \\
				v(x,a,S,t)=0,&\hbox{ in }(0,1)\times (0,A)\times (\eta,T)\\
				v(x,a,s,\eta)=w_{\eta},&\hbox{ in }(0,1)\times (0,A)\times(s_0,S),
			\end{cases}
		\end{equation} 
		where $w_{\eta}=w(x,a,s,\eta)$ in $(0,1)\times(0,A)\times(s_0,S)$. It is quite obvious that the solution $w$ of \eqref{322} satisfies \[w(x,a,s,t)=v(x,a,s,t),\quad\hbox{ a.e. in }(0,1)\times(0,A)\times(s_0,S)\times(0,T).\]
		Since $T>S-s_2$, then there exists $\delta>0$ such that $T>S-s_2+2\delta$. Moreover, if we choose $s_0$ such that $s_0=s_2-\delta,$ we get $$(s_0,S)\subset (s_2-2\delta,S).$$ Therefore, \[T>S-s, \quad \hbox{ for all } \quad s\in (s_0,S),\] and hence \[w(x,a,s,T)=0,\quad\hbox{ a.e. in }\quad(0,1)\times (0,A)\times (s_0,S).\] 
		This complete the proof of Proposition \ref{1p}.
	\end{proof}
	
	After such a long but necessary preparation we can now clearly prove Theorem \ref{Main3}. To this end, let $q=v_1+v_2$ with $v_1$ and $v_2$ are functions satisfying, respectively,
	\begin{equation}\label{3220}
		\begin{cases}
			\dfrac{\partial v_1}{\partial t}-\dfrac{\partial v_1}{\partial a}-\dfrac{\partial v_1}{\partial s}-L v_1+(\mu_1(a)+\mu_2(s))v_1=0,& \hbox{ in }Q_{A,S}\times (\eta,T) ,\\ 
			v_1(0,a,s,t)=v_1(1,a,s,t)=0,&\hbox{ on }(0,A)\times(0,S)\times (\eta,T),,\\ 
			v_1\left( x,A,s,t\right) =0,&\hbox{ in } (0,1)\times(0,S)\times (\eta,T) \\
			v_1(x,a,s_2,t)=0,&\hbox{ in } Q_{S,T}\\
			v_1\left(x,a,s,\eta\right)=q_{\eta},&\hbox{ in }Q_{A,S},
		\end{cases}
	\end{equation} 
	where $q_{\eta}=q(x,a,s,\eta)$ in $Q_{A.S}$, and
	\begin{equation}\label{3221}
		\begin{cases}
			\dfrac{\partial v_2}{\partial t}-\dfrac{\partial v_2}{\partial a}-\dfrac{\partial v_2}{\partial s}-L v_2+(\mu_1(a)+\mu_2(s))v_2=V, &\hbox{ in }Q_{A.S}\times (\eta,T) ,\\ 
			v_2(0,a,s,t)=v_2(1,a,s,t)=0, &\hbox{ on } (0,A)\times(0,S)\times (\eta,T),\\ 
			v_2\left( x,A,s,t\right) =0, &\hbox{ in } (0,1)\times(0,S)\times (\eta,T) \\
			v_2(x,a,s_2,t)=0,&\hbox{ in } Q_{s_2,T}\\
			v_2\left(x,a,s,\eta\right)=0, &\hbox{ in }Q_{A,S}.
		\end{cases}
	\end{equation}
	where $V(x,a,s,t)=\beta(a)q(x,0,s,t)$. According to Duhamel’s formula we clearly have \[v_2(x,a,s,t)=\int_{\eta}^{t}e^{(t-r)\mathcal{A}^{*}}V(\cdot,\cdot,\cdot,r)dr,\]
	where we recall that $\mathfrak{T}^{*}:=(e^{t\mathcal{A}^*})_{t\geq 0}$ is the adjoint semigroup of $\mathfrak{T}$ generated by $\mathcal{A}^{*}$. Moreover, the solution $v_2$ of the system \eqref{3221} satisfies the following estimate.
	\begin{proposition}\label{Propo1}
		Let the assumption of Theorem \ref{Main2} be satisfied. Then, the solution $v_2$ of the system \eqref{3221} satisfies
		\begin{align}
			%		\begin{array}{lll}
			\int\limits_{\eta}^{T}\int\limits_{s_1}^{s_2}\int\limits_{a_1}^{a_2}\int_{\omega}\dfrac{1}{\sigma(x)}v_{2}^2(x,a,s,t)dxdadsdt&\leq \int\limits_{\eta}^{T}\int\limits_{0}^{S}\int\limits_{0}^{A}\int\limits_{0}^{1}\dfrac{1}{\sigma(x)}v_{2}^2(x,a,s,t)dxdadsdt\nonumber\\
			&\leq C \int\limits_{\eta}^{T}\int\limits_{0}^{S}\int\limits_{0}^{1}\dfrac{1}{\sigma(x)}q^2(x,0,s,t)dxdsdt,
			%	\end{array}
		\end{align}
		where $C=e^{\frac{3}{2}T}\|\beta\|^{2}_{\infty}A$. 
	\end{proposition}
	\begin{proof}[Proof of Proposition \ref{Propo1}]
		First, we set $$v_2=\hat{v}_{2}e^{\lambda t},\quad\hbox{ for }\lambda>0.$$ 
		Then the function $v_2$ satisfies 
		\begin{equation}\label{humm}
			\dfrac{\partial \hat{v}_2}{\partial t}-\dfrac{\partial \hat{v}_2}{\partial a}-\dfrac{\partial \hat{v}_2}{\partial s}-L \hat{v}_2+(\mu_1(a)+\mu_2(s)+\lambda)\hat{v}_2=e^{-\lambda t}\beta(a)q(x,0,s,t),\quad\hbox{ in }\quad Q_{A.S}\times (\eta,T). 
		\end{equation}
		Multiplying the above equation by $\dfrac{1}{\sigma}\hat{v}_2$ and integrating over $Q_{A,S}\times (\eta,T)$, we get 
		\[\int\limits_{0}^{S}\int\limits_{0}^{A}\int\limits_{0}^{1}\dfrac{1}{\sigma(x)}\hat{v}_{2}^2(x,a,s,T)dxdads+\int\limits_{0}^{T}\int\limits_{0}^{S}\int\limits_{0}^{1}\dfrac{1}{\sigma(x)}\hat{v}_{2}^2(x,0,s,t)dxdsdt+\int\limits_{0}^{T}\int\limits_{0}^{A}\int\limits_{0}^{1}\dfrac{1}{\sigma(x)}\hat{v}_{2}^2(x,a,0,t)dxdadt\]\[+\int\limits_{\eta}^{T}\int_{Q_{A,S}}\gamma(x)\left(\dfrac{\partial\hat{v}_2}{\partial x}\right)^2dxdadsdt+\int\limits_{\eta}^{T}\int_{Q_{A,S}}(\mu_1(a)+\mu_2(s)+\lambda)\dfrac{1}{\sigma(x)}\hat{v}_{2}^2dxdadsdt\]\[=\int\limits_{\eta}^{T}\int_{Q_{A,S}}e^{-\lambda t}\beta(a)q(x,0,s,t)\dfrac{1}{\sigma(x)}\hat{v}_2dxdadsdt.\]
		Using Young inequality and choosing $\lambda=\frac{3}{2}$, we get
		\[ \int\limits_{\eta}^{T}\int_{Q_{A,S}}\dfrac{1}{\sigma(x)}\hat{v}_{2}^2dxdadsdt\leq \|\beta\|^{2}_{\infty}A \int\limits_{\eta}^{T}\int\limits_{0}^{S}\int\limits_{0}^{1} \dfrac{1}{\sigma(x)}q^2(x,0,s,t)dxdsdt,\]
		and therefore
		\[ \int\limits_{\eta}^{T}\int_{Q_{A,S}}\dfrac{1}{\sigma(x)}v_{2}^2dxdadsdt\leq e^{\frac{3}{2}T}\|\beta\|^{2}_{\infty}A \int\limits_{\eta}^{T}\int\limits_{0}^{S}\int\limits_{0}^{1}\dfrac{1}{\sigma(x)} q^2(x,0,s,t)dxdsdt.\]
	\end{proof}
	
	\begin{proof}[Proof of the Theorem 3.1] We split the proof into two parts. 
		
		\underline{\textbf{Part 1: The case $s_1<a_1+S-s_2$.}}
		
		We split the term to be estimated as follows:
		\[\int\limits_{0}^{S}\int\limits_{0}^{A}\int\limits_{0}^{1}\dfrac{1}{\sigma(x)}q^2(x,a,s,T)dxdads=\int\limits_{0}^{S}\int\limits_{0}^{a_1}\int\limits_{0}^{1}\dfrac{1}{\sigma(x)}q^2(x,a,s,T)dxdads+\int\limits_{0}^{S}\int\limits_{a_1}^{A}\int\limits_{0}^{1}\dfrac{1}{\sigma(x)}q^2(x,a,s,T)dxdads.\]
		According to Proposition \ref{p2}, we have 
		\begin{equation}\label{coupp}
			\int\limits_{0}^{S}\int\limits_{0}^{a_1}\int\limits_{0}^{1}\dfrac{1}{\sigma(x)}q^2(x,a,s,T)dxdads\leq C\int\limits_{0}^{A}\int\limits_{0}^{S}\int\limits_{0}^{A}\int\limits_{0}^{1}\dfrac{1}{\sigma(x)}q^2(x,a,s,t)dxdadsdt.
		\end{equation}
		So it remains to estimate the following term \[\int\limits_{0}^{S}\int\limits_{a_1}^{A}\int\limits_{0}^{1}\dfrac{1}{\sigma(x)}q^2(x,a,s,T)dxdads=\int\limits_{0}^{s_0}\int\limits_{a_1}^{A}\int\limits_{0}^{1}\dfrac{1}{\sigma(x)}q^2(x,a,s,T)dxdads+\int\limits_{s_0}^{S}\int\limits_{a_1}^{A}\int\limits_{0}^{1}\dfrac{1}{\sigma(x)}q^2(x,a,s,T)dxdads.\]
		
		\underline{	\textbf{Upper bound on $(0,s_0)$:}}
		
		Since $q=v_1+v_2$, then we need to estimate 
		\[\int\limits_{0}^{s_0}\int\limits_{a_1}^{A}\int\limits_{0}^{1}\dfrac{1}{\sigma(x)}v_{1}^{2}(x,a,s,T)dxdads+\int\limits_{0}^{s_0}\int\limits_{a_1}^{A}\int\limits_{0}^{1}\dfrac{1}{\sigma(x)}v_{2}^{2}(x,a,s,T)dxdads;\]
		In view of Proposition \ref{Propo1}, we have    \[\int\limits_{0}^{s_0}\int\limits_{a_1}^{A}\int\limits_{0}^{1}\dfrac{1}{\sigma(x)}v_{2}^{2}(x,a,s,T)dxdads\leq C(A-a_2)\int\limits_{\eta}^{T}\int\limits_{0}^{S_0}\int\limits_{0}^{1}\dfrac{1}{\sigma(x)}q^2(x,0,s,t)dxdsdt.\]
		Proposition \ref{Prop1} further yields that
		\begin{equation}\label{debay}
			\int\limits_{0}^{s_0}\int\limits_{a_1}^{A}\int\limits_{0}^{1}\dfrac{1}{\sigma(x)}v_{2}^{2}(x,a,s,T)dxdads\leq C_1\int\limits_{0}^{T}\int\limits_{s_1}^{s_2}\int\limits_{a_1}^{a_2}\int_{\omega}\dfrac{1}{\sigma(x)}q^2(x,a,s,t)dxdadsdt.
		\end{equation}
		\begin{figure}[H]
			\begin{overpic}[scale=0.65]{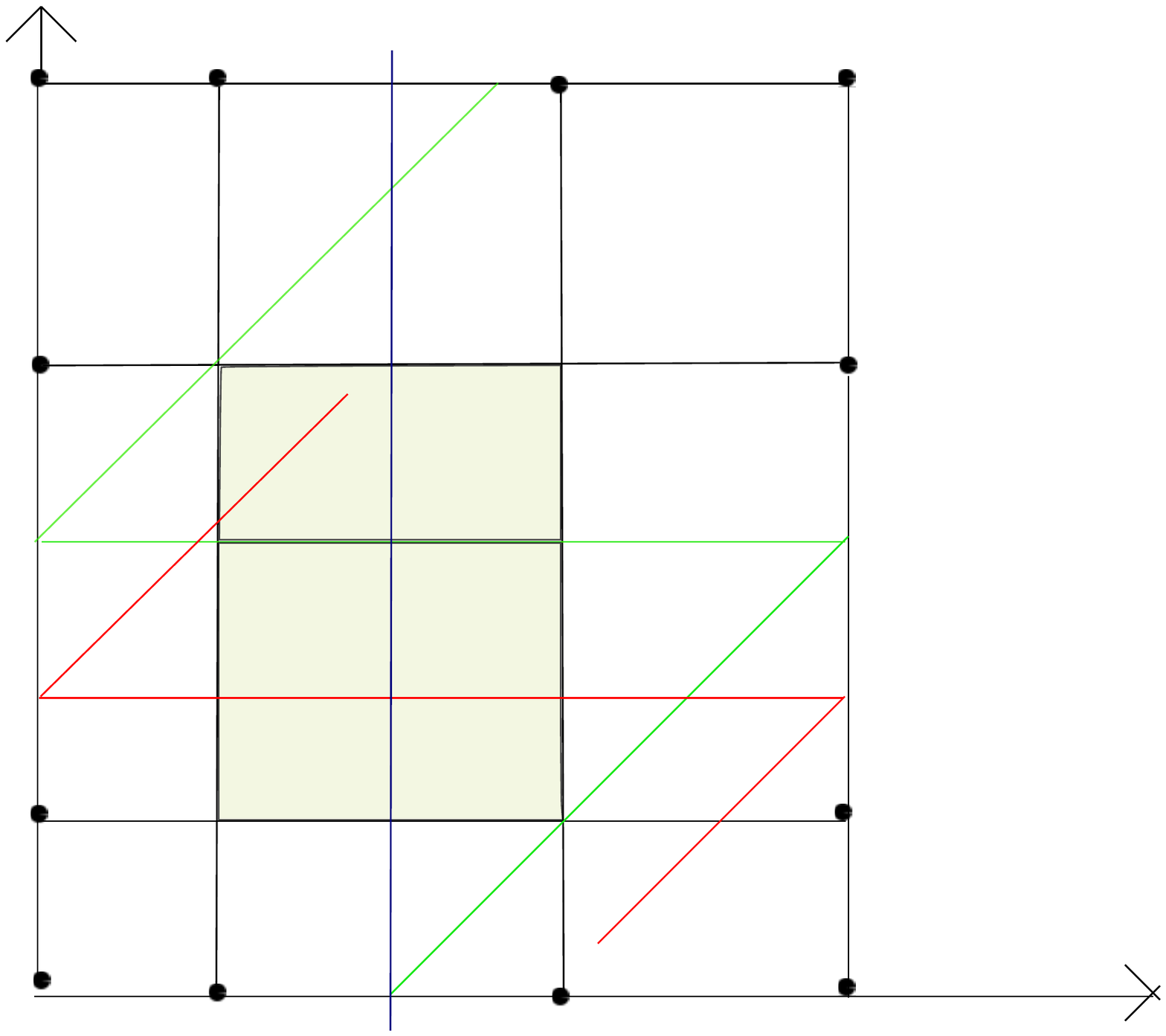}
				\put (26,61) {$S$}	
				\put (69.5,13) {$A$}
				\put (26.5,13.5) {$0$}
				\put (54.5,13) {$a_2$}
				\put (45.3,13) {$a_0$}
				\put (46,15.6) {$.$}
				\put (56.5,18) {$.$}
				\put (26,24.5) {$s_1$}
				\put (26,47) {$s_2$}				
				\put (36.5,13) {$a_1$}
				\put (42.5,35.5) {$Q_{1}$}		
			\end{overpic}
			\caption{Here $\hat{a}=a_2$. The backward characteristics starting from $(a,s,T)$ with $a\in (a_0,A)$ (red line or green line) hits the boundary $(a=A)$, gets renewed by the renewal condition $\beta(a)q(x,0,s,t)$ and then enters the observation domain (red line) or without and leave the domain $\Sigma$ by the boundary $s=S$ (green line).	So, with the conditions $T>A-a_2+S-s_2+a_1+s_1$ all the characteristics starting at $(a,s,T)$ with $a\in (a_0,A)$ get renewed by the renewal condition $\beta(a)q(x,0,s,t)$ with $t>T_1$ and enter the observation domain or without all the domain by the boundary $s=S.$ }
		\end{figure}
		Since $T>A-a_2+s_1+S-s_2$, then there exists $\delta>0$ such that 
		$$T>A-a_2+s_1+S-s_2+2\delta.$$ 
		Therefore,
		$$T-(a_1+S-s_2+\delta)>A-(a_2-(s_1+\delta)).$$
		Moreover, for every $a\in(a_2-(s_1+\delta),A) $, we have  
		\[T-\eta>A-(a_2-(s_1+\delta))>A-a,\] where $\eta=a_1+S-s_2+\delta.$
		Then \[v_1(x,a,s,T)=0\hbox{ a.e. in }\quad(0,1)\times (a_2-(s_1+\delta),A)\times (0,S).\]
		It follows that
		\[\int\limits_{0}^{s_0}\int\limits_{a_1}^{A}\int\limits_{0}^{1}\dfrac{1}{\sigma(x)}v_{1}^2(x,a,s,T)dxdads=\int\limits_{0}^{s_0}\int\limits_{a_1}^{a_2-(s_1+\delta)}\int\limits_{0}^{1}\dfrac{1}{\sigma(x)}v_{1}^{2}(x,a,s,T)dxdads.\]
		Moreover, according Proposition \ref{p2}, we get
		\begin{equation}
			\int\limits_{0}^{S_0}\int\limits_{a_1}^{A}\int\limits_{0}^{1}\dfrac{1}{\sigma(x)}u_{1}^{2}(x,a,s,T)dxdads\leq C\int\limits_{\eta}^{T}\int\limits_{s_1}^{s_2}\int\limits_{a_1}^{a_2}\int_{\omega}\dfrac{1}{\sigma(x)}u_{1}^{2}(x,a,s,t)dxdadsdt.\label{deba}
		\end{equation}
		The fact that $q=v_1+v_2\Longleftrightarrow v_1=q-v_2$ yields  
		\begin{align}\label{deba1}
			\int\limits_{0}^{S_0}\int\limits_{a_1}^{A}\int\limits_0^{1}\dfrac{1}{\sigma(x)}u_{1}^{2}(x,a,s,T)dxdads\leq 2\left(\int\limits_{\eta}^{T}\int\limits_{s_1}^{s_2}\int\limits_{a_1}^{a_2}\int\limits_{l_1}^{l_2}\dfrac{1}{\sigma(x)}q^{2}(x,a,s,t)dQ+\int\limits_{\eta}^{T}\int\limits_{s_1}^{s_2}\int\limits_{a_1}^{a_2}\int\limits_{l_1}^{l_2}\dfrac{1}{\sigma(x)}u_{2}^{2}(x,a,s,t)dQ\right),
		\end{align}
		where we have set $dQ=dxdadsdt$. 
		Thus, by virtue of Proposition \ref{1p} we get 
		\begin{equation}\label{deba11}
			\int\limits_{0}^{s_0}\int\limits_{a_1}^{A}\int\limits_{0}^{1}\dfrac{1}{\sigma(x)}v_{1}^{2}(x,a,s,T)dxdads\leq C\int\limits_{0}^{T}\int\limits_{s_1}^{s_2}\int\limits_{a_1}^{a_2}\int_{\omega}\dfrac{1}{\sigma(x)}q^{2}(x,a,s,t)dxdadsdt.
		\end{equation}
		for a constant $C>0$. The inequalities $(\ref{debay})$ and $(\ref{deba11})$ yield that 
		\begin{equation}
			\int\limits_{0}^{s_0}\int\limits_{a_1}^{A}\int\limits_{0}^{1}\dfrac{1}{\sigma(x)}q^{2}(x,a,s,T)dxdads\leq C_1 \int\limits_{0}^{T}\int\limits_{s_1}^{s_2}\int\limits_{a_1}^{a_2}\int_{\omega}\dfrac{1}{\sigma(x)}q^{2}(x,a,s,t)dxdadsdt.\label{coup1}
		\end{equation}
		for a constant $C_1>0$. Finally, $(\ref{coupp})$ and $(\ref{coup1})$ yield
		\begin{equation}
			\int\limits_{0}^{s_0}\int\limits_{0}^{A}\int\limits_{0}^{1}\dfrac{1}{\sigma(x)}q^{2}(x,a,s,T)dxdads\leq C_T \int\limits_{0}^{T}\int\limits_{s_1}^{s_2}\int\limits_{a_1}^{a_2}\int_{\omega}\dfrac{1}{\sigma(x)}q^{2}(x,a,s,t)dxdadsdt.\label{coup}
		\end{equation}
		
		\underline{	\textbf{Upper bound on $(s_0,S)$: }}
		
		We split the term to be estimated as follows
		\[\int\limits_{s_0}^{S}\int\limits_{0}^{A}\int\limits_0^1\dfrac{1}{\sigma(x)}q^2(x,a,s,T)dxdads=\int\limits_{s_0}^{S}\int\limits_{0}^{a_1}\int\limits_0^1\dfrac{1}{\sigma(x)}q^2(x,a,s,T)dxdads+\int\limits_{s_0}^{S}\int\limits_{a_1}^{A}\int\limits_0^1\dfrac{1}{\sigma(x)}q^2(x,a,s,T)dxdads.\]
		According to Proposition \ref{p2} we have
		\begin{equation}\int\limits_{s_0}^{S}\int\limits_{0}^{a_1}\int\limits_0^1\dfrac{1}{\sigma(x)}q^2(x,a,s,T)dxdads\leq C\int\limits_{0}^{T}\int\limits_{s_1}^{s_2}\int\limits_{a_1}^{a_2}\int_{\omega}\dfrac{1}{\sigma(x)}q^2(x,a,s,t)dxdadsdt.\label{coupp1}\end{equation}
		Since $q=v_1+v_2$ one need to estimate  
		\[\int\limits_{s_0}^{S}\int\limits_{a_1}^{A}\int_{\Omega}\dfrac{1}{\sigma(x)}u_{1}^{2}(x,a,s,T)dxdads+\int\limits_{s_0}^{S}\int\limits_{a_1}^{A}\int_{\Omega}\dfrac{1}{\sigma(x)}u_{2}^{2}(x,a,s,T)dxdads.\]
		Since $T>S-s_2$ and for an appropriate choice of $s_0\in (s_1,s_2)$, by Proposition \ref{p1} we get
		\begin{equation}
			v_2(x,a,s,T)=0\hbox{ a.e. in } (0,1)\times(0,A)\times(s_0,S).
		\end{equation}
		Also, as $T>S-s_2$ then there exists $\delta>0$ such that $T>S-s_2+2\delta$ and then if we take $S_0=s_2-\delta=\eta.$
		Moreover, according to Step 1 of the proof of Proposition \ref{observability-inequality}, we have \[q(x,0,s,t)=0\hbox{ a.e. in }\Omega\times (s_0,S)\times (s_0,T),\]
		Therefore, 
		\[v_2(x,a,s,T)=0\hbox{ a.e. in }\Omega\times (s_0,S)\times (s_0,T).\] Using the Proposition 3.1., we get

		We conclude that 
		
		\begin{equation}
			\int\limits_{S_0}^{S}\int\limits_{0}^{A}\int_{\Omega}\dfrac{1}{\sigma(x)}v^2(x,a,s,T)dxdads\le C_T\int\limits_{0}^{T}\int\limits_{s_1}^{s_2}\int\limits_{a_1}^{a_2}\int_{\omega}\dfrac{1}{\sigma(x)}q^2(x,a,s,t)dxdadsdt.\label{1a}
		\end{equation}
		
		Finally, combining $(\ref{1a})$ and $(\ref{coup})$ we obtain the observability inequality.

		\underline{\textbf{Part 2: the case $s_1>a_1+S-s_2$.}}
		In this part, we need to be able to estimate $q(.,a,s,T)$ if $(a,s)\in (a,s)\in (a_2-s_1,A)\times (0,s_1-a_1)$ in the case where $T>A-a_2+S-s_2+a_1+s_1.$\\ For it, we split this section as follows: \begin{align*}S_{T}^{1}=\{(a,s,T)\hbox{ such that }0<s<s_1-a_1 \hbox{ and } s+A-s_1+a_1<a<A\}\end{align*} and
		\begin{align*}S_{T}^{2}=\{(a,s,T)\hbox{ such that }0<s<s_1-a_1 \hbox{ and } a_2-s_1<a<A-s_1+a_1+s\}.\end{align*}
		% \[S_{T}^{3}=\{(a,s,T)\hbox{ such that }0<s<s_1-a_1 \hbox{ and } a_2<a<s+a_2\}.
		All the backwards characteristics starting $(a,s,T)$ with $(a,s,T)\in S_T$ get renewed by the renewal term with $a\in (A-\delta,A)$ and need at least $s_1-\delta$ to enter the observation domain.\\
		In $S^{1}_{T}$ we have
		\[q(x,a,s,T)=\int\limits_{T-A+a}^{T}\beta(a+T-l)q(x,0,s+T-l,l)dl\]
		\[=\int\limits_{a}^{A-a}\beta(\sigma)q(x,0,s+\sigma-a,a+T-\sigma)d\sigma.\]
		%\[\leq A\|\beta\|^{2}_{\infty}\int\limits_{a}^{A-a}q^2(x,0,s+\sigma-a,a+T-\sigma)d\sigma.\]
		Then
		\[\int_{\Omega}\dfrac{1}{\sigma}q^2(x,a,s,T)\leq A\|\beta\|^{2}_{\infty}\int\limits_{a}^{A-a}\int_{\Omega}\dfrac{1}{\sigma}q^2(x,0,s+\sigma-a,a+T-\sigma)d\sigma\]
		Let $\delta>0$ be such that $T>s_1+\delta.$ Subdividing the interval $(0,s_1-a_1)$ in $M$ intervals as follows:
		\[0=\delta_0<\delta_1<\delta_2<\delta_3<...<\delta_M\hbox{ such that }\delta_i-\delta_{i+1}<\delta\hbox{, } i\in\{0,1,...,M-1\},\] with
		\[\delta_0=0\hbox{ and }\delta_M=s_1-a_1\]
		We denote by \[S_{T,i}^{1}=\{(a,s,T)\hbox{ such that }\delta_i<s<\delta_{i+1} \hbox{ and } s+A-\delta_{i+1}<a<A-\delta_i\}.\]
		Moreover, in $S_{T,i},$ \[\max_{S_{T,i}^{1}}\{T-S+s,T-A+a\}>T-\delta_{i+1}+s.\]
		Therefore, if $(a,s,T)\in S^{1}_{T,i}\hbox{ and }T>s_1+\delta,$ then $T-\delta_{i+1}>s_1-\delta_i$
		Moreover, for every $\eta$ and $T$ such that $s_1-\delta_{i}<\eta<T,$ we have \[\int\limits_{\eta}^{T}\int\limits_{\delta_i}^{\delta_{i+1}}\int_{\Omega}\dfrac{1}{\sigma}q^2(x,0,s,t)dxdsdt\leq C_{T}\int\limits_{0}^{T}\int\limits_{s_1}^{s_2}\int\limits_{a_1}^{a_2}\int_{\omega}\dfrac{1}{\sigma}q^2(x,a,s,t)dxdadsdt.\] Therefore,
		\[\int\limits_{0}^{T}\int\limits_{\delta_i}^{\delta_{i+1}}\dfrac{1}{\sigma}q^2(x,a,s,T)dxads\leq C_T\int\limits_{0}^{T}\int\limits_{s_1}^{s_2}\int\limits_{a_1}^{a_2}\int_{\omega}\dfrac{1}{\sigma}q^2(x,a,s,t)dxdadsdt.\] Then for $T>s_1,$ we have
		\[\int_{S_{T}^{1}}\dfrac{1}{\sigma}q^2(x,a,s,T)dxdadt\leq C_T\int\limits_{0}^{T}\int\limits_{s_1}^{s_2}\int\limits_{a_1}^{a_2}\int_{\omega}\dfrac{1}{\sigma}q^2(x,a,s,t)dxdadsdt.\]
		\begin{figure}[H]
			\begin{overpic}[scale=0.65]{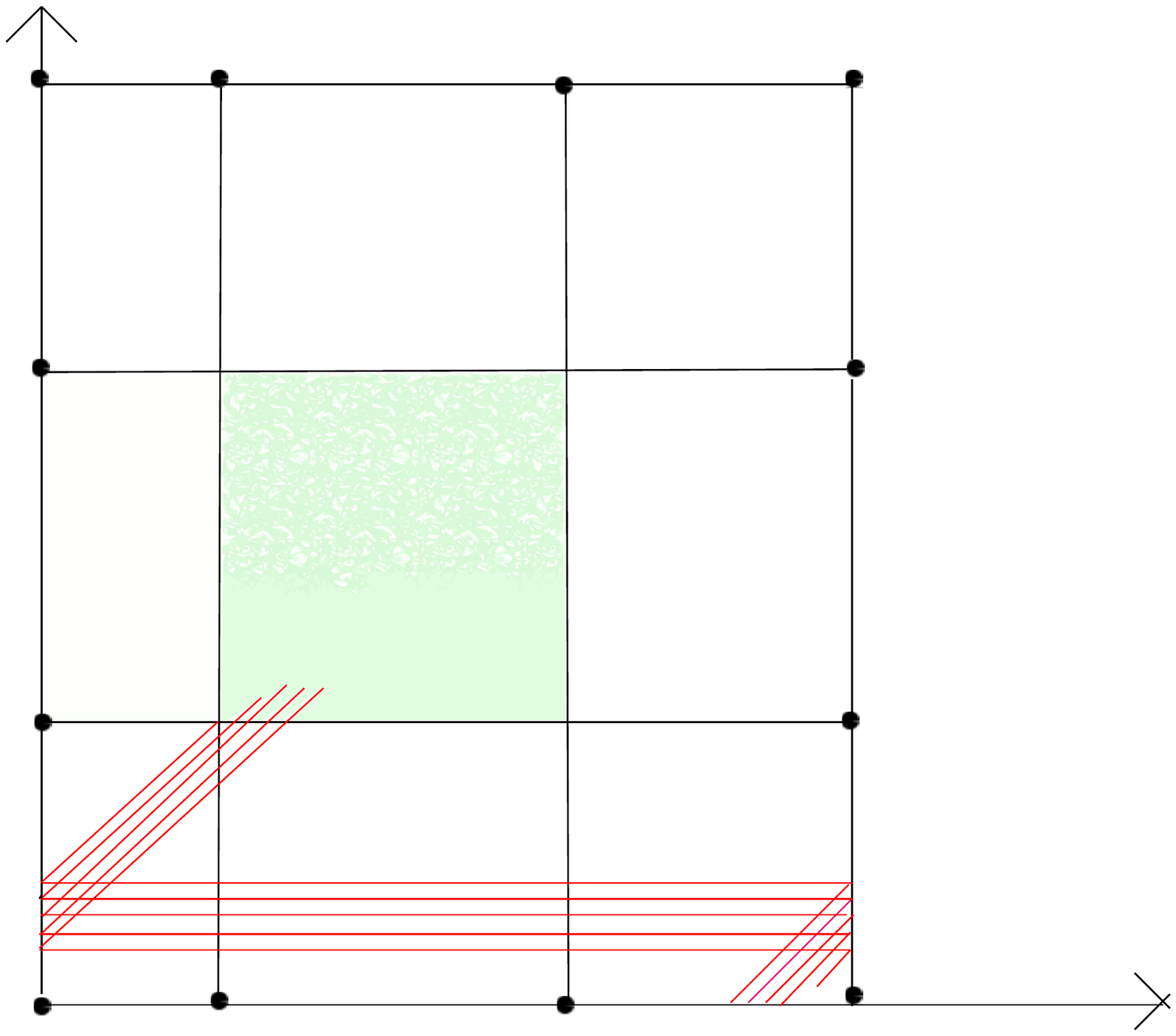}
				\put (26,61) {$S$}	
				\put (69.5,13) {$A$}
				\put (26.5,13.5) {$0$}
				\put (54.5,13) {$a_2$}
				\put (63,13.5) {$b$}
				\put (63,15) {$\bullet$}
				\put (26,29) {$s_1$}
				\put (28,21.25) {$\bullet$}
				\put (21.5,21.5) {$s_1-a_1$}
				\put (26,47) {$s_2$}			
				\put (36.5,13) {$a_1$}
				\put (42.5,35.5) {$Q_{1}$}	
				\put(50,5){$b=A-s_1+a_1$}
			\end{overpic}
			\caption{Illustration in the case $s_1>a_1$ \\In the case where $s_1>a_1$, only the backwards characteristics starting at $(a,s,T)\in S_T$ get renewed by the renewal term with $s\in(0,s_1-a_1)$}	
		\end{figure}
	\end{proof}

\end{document}